\documentclass[10pt]{article}
\usepackage{amssymb,amsmath,amsthm,amsfonts}
\usepackage{cite}
\usepackage{bigints}
\usepackage[colorlinks=true,linkcolor=blue,urlcolor=blue,citecolor=blue]{hyperref}
\usepackage[margin=100pt]{geometry}

\usepackage[pagewise,mathlines]{lineno}
\usepackage{fancyhdr}
\tt
\theoremstyle{plain}
\newtheorem{theorem}{Theorem}
\newtheorem{lemma}{Lemma}
\newtheorem{proposition}{Proposition}
\newtheorem{corollary}{Corollary}
\theoremstyle{definition}
\newtheorem{definition}{Definition}
\theoremstyle{remark}
\newtheorem{remark}{Remark}

\newcommand{\esssup}{{\mathrm{ess}\sup\,}}
\newcommand{\essinf}{{\mathrm{ess}\inf\,}}

\begin{document}

\title{\bf\Large Optimization of the principal eigenvalue\\ of the Neumann Laplacian with indefinite weight \\ and monotonicity of minimizers in cylinders}

\author{Claudia Anedda\footnote{Department of Mathematics and Computer Science,
University of Cagliari, Via Ospedale 72, Cagliari, 09124,  Italy (\tt canedda@unica.it).}\; and\;
Fabrizio Cuccu\footnote{Department of Mathematics and Computer Science,
University of Cagliari, Via Ospedale 72, Cagliari, 09124,  Italy (\tt fcuccu@unica.it).}}

\maketitle
\begin{abstract}
 \noindent Let $\Omega\subset\mathbb{R}^N$, $N\geq 1$, be an open bounded connected set. We consider the indefinite 
weighted  eigenvalue problem   $-\Delta u =\lambda m u$ in $\Omega$ with $\lambda \in \mathbb{R}$,
$m\in L^\infty(\Omega)$ and with homogeneous Neumann boundary conditions. We study weak* continuity, 
convexity and G\^ateaux differentiability of   
the map $m\mapsto1/\lambda_1(m)$, where $\lambda_1(m)$ is the principal eigenvalue. 
Then, denoting by $\mathcal{G}(m_0)$ the class of rearrangements of a fixed weight $m_0$,
under the assumptions that $m_0$ is positive on a set of positive Lebesgue measure and $\int_\Omega m\,dx<0$, 
we prove the existence and a characterization of minimizers of $\lambda_1(m)$ and the non-existence of
maximizers.    
Finally, we show that, if $\Omega$ is a cylinder, then every minimizer is
monotone with respect to the direction of the generatrix. 
In the context of the population dynamics, this kind of problems arise from the 
question of determining the optimal spatial location of favourable and unfavourable habitats 
for a population to survive.  
\end{abstract}



\section{Introduction and main results}\label{intro}

In this paper we consider the weighted eigenvalue problem with homogeneus Neumann boundary conditions  
\begin{equation}\label{p0}
\begin{cases}-\Delta u =\lambda m u \quad &\text{in } \Omega\\
\cfrac{\partial u}{\partial \nu}=0  &\text{on } \partial\Omega, \end{cases}  
\end{equation}  
where $\Omega\subset\mathbb{R}^N$ is an open bounded domain with Lipschitz boundary $\partial\Omega$, 
 $m\in L^\infty(\Omega)$ changes sign in $\Omega$, $\lambda\in\mathbb{R}$
and $\nu$ is the outward unit normal vector on $\partial \Omega$.\\  
An eigenvalue $\lambda$ of \eqref{p0} is called {\it principal eigenvalue} if it admits a positive eigenfunction. Clearly, $\lambda =0$ is a principal eigenvalue with positive constants as its eigenfunctions. Problem \eqref{p0} has been studied in various papers (see, for example, \cite{Bo,BL,SH}). In particular, it is known that there is a positive (respectively negative) principal eigenvalue if and only if $\int_\Omega m\, dx<0$ (respectively $\int_\Omega m\, dx>0$). For the sake of completeness and in order to maintain this paper self-contained, we prefer to give here
(see Section \ref{preliminaries}) an independent proof of the result above. Moreover, we show that, under the previous hypothesis on $m$, there exists an increasing (respectively decreasing) sequence of     
positive (respectively negative) eigenvalues. The smallest positive eigenvalue is the principal eigenvalue, which will be denoted by $\lambda_1(m)$.\\   
Problem \eqref{p0} and its variants play a crucial role in studying nonlinear models from 
population  dynamics (see \cite{S}) and population genetics (see \cite{F}).           
We illustrate in details the following model in population dynamics devised by Skellam \cite{S}
\begin{equation}\label{p1}
\begin{cases}
v_t=\Delta v+\gamma v[m(x)-v] \quad &\text{in }\Omega\times(0,\infty),\\
\cfrac{\partial v}{\partial \nu}\,=0 & \text{on } \partial\Omega\times(0,\infty),\\
v(x,0)=v_0(x)\geq 0, \  v(x,0)\not \equiv 0& \text{in } \overline{\Omega}.
\end{cases}  
\end{equation}

In \eqref{p1} $v(x,t)$ represents the density of a population inhabiting the region $\Omega$ at location $x$ and time $t$ (for that reason, only non-negative solutions of \eqref{p1} are of interest), $v_0$ is the initial density and $\gamma$ is a positive parameter. The function $m(x)$ represents the intrinsic local grow rate of
the population, it is positive on favourable habitats and negative on 
unfavourable ones and it mathematically describes the available resources in the spatially heterogeneous environment
$\Omega$. The integral $\int_\Omega m\, dx$ can be interpreted as a measure of the total resources in $\Omega$. The Neumann conditions in \eqref{p1} are zero-flux boundary conditions: it means that no individuals cross the boundary of the habitat, i.e. the boundary acts as a barrier.\\
The model \eqref{p1} can be also considered with Neumann boundary conditions replaced by the homogeneous Dirichlet or Robin conditions. In the first case, the environment $\Omega$ is surrounded by a completely inhospitable region, i.e. any individual reaching the boundary dies, while in the second some individuals reaching the boundary die
and the others return to the interior of $\Omega$.  
It is known (see \cite{CC} and references therein) that \eqref{p1} predicts persistence for the population if $\lambda_1(m)<\gamma$. As a consequence, determining the best spatial arrangement of favourable and unfavourable habitats for the survival, within a fixed class of environmental configurations, results in minimizing $\lambda_1(m)$ over the corresponding class of weights. Having information of this type could affect, for example, on the strategies to be adopted for the conservation of species with limited resources.\\  
This kind of problem has been investigated by many other authors. The question of determining the optimal spatial     arrangements of favourable and unfavourable habitats in $\Omega$ for the survival of the modelled population
was first addressed by Cantrell and Cosner in \cite{CC, CC91}. The authors considered the 
diffusive logistic equation \eqref{p1} with homogeneous Dirichlet boundary conditions and when the weight $m$ has fixed maximum, minimum and integral over $\Omega$. The analogous problem with Neumann boundary conditions has been analysed
by Lou and Yanagida in \cite{LY}. Berestycki et al. \cite{BHR} investigated a model similar to \eqref{p1} in the case of periodically fragmented environment ($\Omega=\mathbb{R}^N$ and   
$m(x)$ periodic), Roques and Hamel \cite{RH} studied the optimal arrangement of resources by using numerical 
computation, Jha and Porru \cite{JP}, among other things, exhibited an example of symmetry breaking of the optimal 
arrangement of the local growth rate. Lamboley et al. \cite{LLNP} investigated model \eqref{p1} with Robin boundary 
conditions. Mazari et al. \cite{MNP} studied several shape optimization problems arising in population dynamics,
we refer the reader to it for a review of current knowledge on the subject.
We also mention Cadeddu et al. \cite{CFP}, which considered mixed boundary conditions, Derlet et al. \cite{DGT}, that 
extended these type of results to the principal eigenvalue associated to 
the $p$-Laplacian operator,  Pellacci and Verzini \cite{PV} to the fractional Laplacian operator and  Dipierro et al. 
\cite{DPLV} to a mixed 
local and nonlocal operator. \\ 
In order to present our work, we briefly give some notations and definitions here. 
We denote by $\lambda_k(m)$, $k\in \mathbb{N}$, the $k$-th positive eigenvalue of problem
 \eqref{p0} corresponding to the weight $m$ (assuming $ \int_\Omega m\,dx<0$). We say that
 two Lebesgue measurable functions $f,g:\Omega \to \mathbb{R}$
are \emph{equimeasurable} if the superlevel sets
$\{x\in \Omega: f(x)>t\}$ and $\{x\in \Omega: g(x)>t\}$ have the same  measure for all $t\in \mathbb{R}$. For a fixed $f\in L^\infty(\Omega)$, we call the set
$\mathcal{G}(f)=\{g:\Omega\to\mathbb{R}: g \text{ is measurable and $g$ and
$f$ are equimeasurable}\}$ the
\emph{class of rearrangements of $f$} (see Subsection
 \ref{rearrangements}).
Moreover, we introduce the set 
$L^\infty_<(\Omega)=\left\{m\in L^\infty (\Omega): \int_\Omega m\,dx <0\right\}$. \\
 The present paper contains three main results. First, we study the dependence of $\lambda_k(m)$ on $m$, in 
 particular we investigate
 continuity and, for $k=1$, convexity and differentiability properties (see Lemmas \ref{teo1ii},  \ref{teo2} and \ref{teo3}).
Second, we examine the optimization of $\lambda_1(m)$ in the class of rearrangements $\mathcal{G}(m_0)$
 of a fixed function $m_0\in L^\infty_<(\Omega)$.
Precisely, we prove the existence of minimizers, a characterization of them in terms of the 
 eigenfunctions relative to $\lambda_1(m)$ and a non-existence result for the maximizers.
\begin{theorem}\label{exist}
Let $\lambda_1(m)$ be the principal eigenvalue of problem
\eqref{p0},   
$m_0\in L_<^\infty(\Omega)$    
such that the set $\{x\in \Omega:m_0(x)>0\}$ has positive Lebesgue measure, $\mathcal{G}(m_0)$ the class of rearrangements
of $m_0$ (see Definition \ref{class}) and $\overline{\mathcal{G}(m_0)}$ its weak* closure in $L^\infty(\Omega)$.
Then \\    
i) the problem   
\begin{equation}\label{infclos0}
\min_{m\in{\overline{\mathcal{G}(m_0)}}} \lambda_1(m)
\end{equation} 
admits solutions and any solution $\check{m}_1$ belongs to $\mathcal{G}(m_0)$;\\
ii) for every solution $\check{m}_1\in\mathcal{G}(m_0)$ of \eqref{infclos0}, 
 there exists an increasing function $\psi$ such
that      
\begin{equation}\label{carat}\check{m}_1= \psi(u_{\check{m}_1}) \quad \text{a.e. in }\Omega, \end{equation} where 
$u_{\check{m}_1}$ is 
the unique positive eigenfunction relative to $\lambda_1(\check{m}_1)$ normalized 
as in \eqref{normaliz1};\\
iii)
\begin{equation*}
\sup_{m\in{\mathcal{G}(m_0)}} \lambda_1(m)=+\infty.
\end{equation*}
\end{theorem}    

We note that the class of weights usually considered in literature, i.e. a set of bounded functions with fixed
maximum, minimum and integral over $\Omega$, can be written as $\overline{\mathcal{G}(m_0)}$ for a
$m_0$ which takes exactly two values (functions of this kind are called of ``bang-bang'' type). 
This fact is proved in \cite{AC}.\\
From the biological point of view, i) of Theorem 1 says that there exists an arrangement of the 
resources  that maximizes the chances of survival and, in this case, the population density is
larger where the habitat is more favourable. On the other hand, iii) means that there are 
configurations of resources as bad (i.e. inhospitable) as one prescribes.\\
Our third main result is the following 
\begin{theorem}\label{steiner}
Let $\Omega=(0,h)\times\omega\subset\mathbb{R}^N$, $h>0$ and $\omega\subset\mathbb{R}^{N-1}$ be a bounded smooth domain.  Let ${m}_0\in L_<^\infty(\Omega)$ such that the set $\{x\in \Omega:m_0(x)>0\}$ has positive Lebesgue measure and $\mathcal{G}(m_0)$ the class of rearrangements
of $m_0$ (see Definition \ref{class}). Then every minimizer of \eqref{infclos0}
is monotone with respect to $x_1$, where $x_1$ is the first coordinate of $\mathbb{R}^N$. 
\end{theorem}          

Monotonicity results of this kind have been studied both theoretically and numerically by a 
number of authors. 
Theorem \ref{steiner} in the one dimensional case has been proved in \cite{CC91, LY}
in the case $m_0$ is a ``bang-bang'' function and in \cite{JP} for general $m_0$.      
In general dimension, when the domain $\Omega$ is an orthotope and $m_0$ is of ``bang-bang'' type,
Lamboley et al. in \cite{LLNP} show  
that any minimizer is monotonic with respect to every coordinate direction.
Theorem \ref{steiner} contains all previous results and it is coherent with numerical simulations
in \cite{RH, KLY} for rectangles and ``bang-bang'' weights.  \\ 
It is worth mentioning that  in the case \eqref{p0} is considered with Dirichlet boundary conditions, the 
monotonicity of minimizers is replaced by the Steiner symmetry of them (see \cite{CC91, AC}). 
Nevertheless, in both situations these qualitative properties of the minimizers lead to an
arrangement of the favourable resources fragmented as little as possible.  Indeed, in the 
Dirichlet case they are concentrated far from the boundary, while in the Neumann
case they meet the boundary.\\  
As final remark, we observe that  problem \eqref{p0}, with Dirichlet boundary conditions in place of
Neumann and in the case of positive weight $m(x)$,   
also has a well known physical interpretation: it models the normal modes of vibration of a membrane
$\Omega$ with clamped boundary $\partial\Omega$ and mass
density $m(x)$; $\lambda_1(m)$ represents the principal natural frequency of the membrane. 
Therefore, physically, minimizing $\lambda_1(m)$ means to find the mass distribution of the 
membrane which gives the lowest principal natural frequency. 
Usually, the composite membrane is built using only two homogeneous materials of 
different densities  and, then, the weights in the optimization problem take only two positive 
values. Among many papers that consider the optimization of the principal natural frequency, we recall 
\cite{CGIKO,CML1,CML2}. \\
This paper is structured as follows. In Section \ref{preliminaries} we set up the functional framework 
and some tools in order to investigate the spectrum of problem \eqref{p0}.
In Section \ref{rearrangements} we collect some known results about rearrangements of 
measurable functions we will need in the sequel to examine
the optimization problem \eqref{infclos0}.  
In Section \ref{existence} we study the dependence of $\lambda_k(m)$ on $m$, in particular
continuity and, for $k=1$, convexity and differentiability properties; then, we prove 
 Theorem \ref{exist}.        
 Finally, in Section \ref{monotonicity} we give the proof of Theorem \ref{steiner}.

\section{Notations, preliminaries and weak formulation of \eqref{p0}}\label{preliminaries}   
Let $\Omega \subset \mathbb{R}^N$, $N\geq 1$, be a bounded connected open set with Lipschitz boundary 
$\partial \Omega$.\\
In this paper, we denote by $|E|$ the measure of an arbitrary Lebesgue measurable set $E \subset \mathbb{R}^N$ and by
$L^\infty(\Omega)$, $L^2(\Omega)$ and $H^1(\Omega)$ the usual Lebesgue and Sobolev spaces. The usual norms and scalar 
products of these spaces are denote by
\begin{equation*}
\|u\|_{L^\infty(\Omega)}=\esssup_{\Omega} |u| \quad \forall\, u\in L^\infty(\Omega),       
\end{equation*}    
\begin{equation*}
\langle u,v \rangle_{L^2(\Omega)}=\int_\Omega uv \, dx \quad \forall\, u, v\in L^2(\Omega)
\end{equation*}     
\begin{equation*}
\|u\|_{L^2(\Omega)}=\langle u,u \rangle^{1/2}_{L^2(\Omega)} \quad \forall\, u\in L^2(\Omega),        
\end{equation*}  
\begin{equation*}
\langle u,v \rangle_{H^1(\Omega)}=\int_{\Omega}uv\, dx +\int_\Omega \nabla u \cdot \nabla v \, dx 
\quad \forall\, u, v\in H^1(\Omega),
\end{equation*}
\begin{equation}\label{normah1}
\|u\|_{H^1(\Omega)}=\langle u,u \rangle_{H^1(\Omega)}^{1/2} \quad \forall\, u\in H^1(\Omega).
\end{equation}  
Moreover, we also use the notation $\langle \nabla u,\nabla v \rangle_{L^2(\Omega)}=\int_\Omega \nabla u\cdot \!\nabla v \, dx$
for all $u, v\in H^1(\Omega)$ and
by weak* convergence we always mean the weak* convergence in $L^\infty(\Omega)$.\\  
Given $m\in L^2(\Omega)$ such that $m \neq 0$,  we define the spaces
$$L^2_m(\Omega) = \left\{ f\in L^2(\Omega): \int_\Omega m f\, dx =0\right\}\quad
\text{and}\quad V_m(\Omega) = H^1(\Omega)\cap L^2_m(\Omega).$$
$L^2_m(\Omega)$ and $V_m(\Omega)$ are separable Hilbert subspaces of $L^2(\Omega)$ and 
$H^1(\Omega)$ respectively.\\ 
      
\subsection{The projection $P_m$ and norm in $V_m(\Omega)$}      
   
In this subsection we introduce a fundamental tool in order to develop our theory: a projection from $L^2(\Omega)$ to $L^2_m(\Omega)$ (which must not be confused with the usual orthogonal projection in Hilbert spaces).
      
\begin{definition}\label{def1}
Let $m\in L^2(\Omega)$ such that $\int_\Omega m\, dx \neq 0$. We call 
$\it{projection}$ $P_m$ the operator
\begin{equation*} 
 P_m: L^2(\Omega) \to L^2_m(\Omega),\qquad  f\mapsto f- \frac{\int_\Omega mf \, dx}{\int_\Omega m\, dx}\,.
\end{equation*} 
\end{definition}

Note that $P_m(H^1(\Omega))\subset V_m(\Omega)$. Indeed, depending on the case (which will be clear from the context), it might be more convenient to consider the projection $P_m: H^1(\Omega) \to V_m(\Omega)$.   
Since in our work $m(x)$ represents the local growth rate, which is a bounded function, hereafter we consider $m\in L^\infty(\Omega)$. Nevertheless, Proposition \ref{prop1}, Proposition  \ref{prop2} and Proposition \ref{normav} can also be stated for $m\in L^2(\Omega)$.

\begin{proposition}\label{prop1}
Let $m,q\in L^\infty(\Omega)$ such that $\int_\Omega m\, dx, \int_\Omega q\, dx \neq 0$ and $P_m$ the projection of   
Definition \ref{def1}.   Then \\ 
i) $\langle mP_m(f), \varphi\rangle_{L^2(\Omega)}= \langle m f, P_m(\varphi)\rangle_{L^2(\Omega)}$  
for all  $f,\varphi\in L^2(\Omega)$;\\
ii)  $P_m(f)=0$ if and only if $f$ is constant;\\   
iii) $P_m(f)=f$ for all $f\in L_m^2(\Omega)$;\\      
iv) $\nabla P_m(f)=\nabla f$ for all $f\in H^1(\Omega)$;\\  
v) $P_m$ is a linear bounded operator with  
\begin{equation}\label{stima1}
\|P_m\|_{\mathcal{L}(L^2(\Omega),L^2(\Omega)) }\leq 1+ \frac{\|m\|_{L^\infty(\Omega)}}{\left|\int_\Omega m\, dx \right|}\, |\Omega|
\end{equation}
and
\begin{equation}\label{stima2}
\|P_m\|_{\mathcal{L}(H^1(\Omega),H^1(\Omega))}\leq 1+ \frac{\|m\|_{L^\infty(\Omega)}}{\left|
\int_\Omega m\, dx \right|}\, |\Omega|;
\end{equation}  
vi) the compositions $P_q\circ P_m: L^2_q(\Omega) \to L^2_q(\Omega)$,  $P_q\circ P_m: 
V_q(\Omega) \to V_q(\Omega)$ are identities;\\  
vii) $ P_m: L^2_q(\Omega) \to L^2_m(\Omega)$,  $P_m: V_q(\Omega) \to V_m(\Omega)$ are 
isomorphisms. 
\end{proposition} 
 
\begin{proof}      
i), ii), iii) and iv) are immediate consequences of the definition of the projection $P_m$.\\
v) By the definition of $P_m$ and straightforward calculations we find          
 \begin{equation*}                          
 \begin{split}  
 \|P_m(f)\|^2_{L^2(\Omega)} &=
 \bigintsss_\Omega\left[f^2-2\,\frac{\int_\Omega mf\;dx}{\int_\Omega m\;dx}f+\left(\frac{\int_\Omega mf\;dx}{\int_\Omega m\;dx}\right)^2 \right]\;dx\\
 &= \|f\|_{L^2(\Omega)}^2-2\,\frac{\int_\Omega mf\;dx}{\int_\Omega m\;dx}\int_\Omega f\;dx+\left(\frac{\int_\Omega mf\;dx}{\int_\Omega m\;dx}\right)^2|\Omega|\\
&\leq  \|f\|_{L^2(\Omega)}^2+2\,\frac{\|m\|_{L^2(\Omega)}\|f\|^2_{L^2(\Omega)}}{\left|\int_\Omega m\;dx
\right|}|\Omega|^{1/2}+\,\frac{\|m\|^2_{L^2(\Omega)}\|f\|^2_{L^2(\Omega)}}{\left|\int_\Omega m\;dx
\right|^2}|\Omega|\\
&=\left(\|f\|_{L^2(\Omega)}+\,\frac{\|m\|_{L^2(\Omega)}\|f\|_{L^2(\Omega)}}{\left|\int_\Omega m\;dx\right|}|\Omega|^{1/2}\right)^2\\
&\leq \left(\|f\|_{L^2(\Omega)}+\,\frac{\|m\|_{L^\infty(\Omega)}\|f\|_{L^2(\Omega)}}{\left|\int_\Omega m\;dx\right|}|\Omega|\right)^2\\
& = \left(1+\,\frac{\|m\|_{L^\infty(\Omega)}}{\left|\int_\Omega m\;dx\right|}|\Omega|\right)^2\|f\|^2_{L^2(\Omega)};\\
\end{split}
\end{equation*}        
then \eqref{stima1} holds. The estimate \eqref{stima2} follows from \eqref{stima1} and iv).\\ 
vi) Let $f\in  L^2_q(\Omega)$;  recalling that $\int_\Omega qf\, dx = 0$, we have
$$P_q (P_m(f)) = P_q \left(f- \frac{\int_\Omega mf \, dx}{\int_\Omega m\, dx} \right)
= f- \frac{\int_\Omega mf \, dx}{\int_\Omega m\, dx}\, - 
\frac{ \int_\Omega qf\,dx- \frac{\int_\Omega mf \, dx}{\int_\Omega m\, dx}\int_{\Omega}q\, dx}{\int_\Omega q\,dx} =f;$$
the second statement immediately follows from the first one.\\  
vii) It immediately follows from vi).  
\end{proof}

For the sake of convenience, we put
\begin{equation}\label{c1} 
 C_1(m)=1+ \frac{\|m\|_{L^\infty(\Omega)}}{\left|\int_\Omega m\, dx \right|}\, |\Omega|.
\end{equation} 

The previous proposition leads us to an alternative norm in the space $V_m(\Omega)$.

\begin{proposition}\label{prop2}
Let $m\in L^\infty(\Omega)$ such that $\int_\Omega m\, dx\neq 0$. Then, for all $u\in V_m(\Omega)$ we have
\begin{equation}\label{stima3}
 \|u\|_{L^2(\Omega)}\leq C\left(1+ \frac{\|m\|_{L^\infty(\Omega)}}{\left|\int_\Omega m\, dx \right|}\,
 |\Omega|\right)\|\nabla u\|_{L^2(\Omega)},
\end{equation}
where $C$ is the constant of the Poincar\'e-Wirtinger's inequality (see \cite[Theorem 12.23]{L}).    
\end{proposition}        
\begin{proof} Let $u\in V_m(\Omega)$.  By vi) of Proposition \ref{prop1} we have $u=P_m(P_1(u))$. 
By \eqref{stima1} and the Poincar\'e-Wirtinger's inequality we find
\begin{equation*}
 \begin{split}   
\|u\|_{L^2(\Omega)}=\|P_m(P_1(u))\|_{L^2(\Omega)}&\leq C_1(m)\|P_1(u)\|_{L^2(\Omega)}\\
&\leq C_1(m)\left \|u - \frac{1}{|\Omega|} \int_\Omega u\, dx\right\|_{L^2(\Omega)}\leq C_1(m)\cdot C \|\nabla u \|_{L^2(\Omega)},
\end{split}
\end{equation*}
which proves the statement.
\end{proof}

\begin{proposition}\label{normav}    
Let $m\in L^\infty(\Omega)$ such that $\int_\Omega m\, dx\neq 0$. Then, the bilinear form in $V_m(\Omega)$
\begin{equation}\label{prodscalvm} \langle u,v \rangle_{V_m(\Omega)}=\int_\Omega \nabla u \cdot \nabla v\, dx \quad \forall\, u, v\in V_m(\Omega)   
\end{equation}
is a scalar product which induces a norm equivalent to the usual
 norm \eqref{normah1}.
We denote by $\|u\|_{V_m(\Omega)}$ the associated norm to \eqref{prodscalvm}.
\end{proposition}
\begin{proof}
Comparing $\|u\|_{V_m(\Omega)}$ with \eqref{normah1}, we have  
\begin{equation}\label{stimanorme1}
\|u\|_{V_m(\Omega)}\leq \|u\|_{H^1(\Omega)}.     
\end{equation}
By \eqref{stima3} and \eqref{c1}, we find 
\begin{equation}\label{stimanorme2}
\|u\|_{H^1(\Omega)}\leq (C^2\cdot C_1^2(m)+1)^{1/2}\|u\|_{V_m(\Omega)}.          
\end{equation}   
By \eqref{stimanorme1} and \eqref{stimanorme2}, the thesis immediately follows.   
\end{proof}

If not stated otherwise, we will consider $V_m$ endowed with the norm just introduced.

\subsection{The operators $E_m$ and $G_m$}\label{operators}

We study the eigenvalues of problem \eqref{p0} by means of the spectrum of an operator that we will introduce in this subsection.\\
Let $m\in L^\infty(\Omega)$ such that $\int_\Omega m\, dx\neq 0$. For every $f\in L^2(\Omega)$ let us consider the following continuous linear functional on $V_m(\Omega)$ 
\begin{equation*} 
\varphi\mapsto
 \langle m f,\varphi\rangle_{L^2(\Omega)} \quad\forall\, \varphi\in V_m(\Omega).
\end{equation*}
By the Riesz Theorem, 
there exists a unique $u\in V_m(\Omega)$ such that 
\begin{equation}\label{volpe}
\langle  u, \varphi\rangle_{V_m(\Omega)}
= \langle m f,\varphi\rangle_{L^2(\Omega)} \quad\forall\, \varphi\in V_m(\Omega)
\end{equation}
holds.\\
Let us introduce the operator
\begin{equation}\label{Em} 
 E_m:L^2(\Omega)\to V_m(\Omega),
\end{equation} 
where $u=E_m(f)$ is the unique function in $V_m(\Omega)$ that satisfies \eqref{volpe}, i.e.  for
all $f\in L^2(\Omega)$, $E_m(f)$ is defined by 
\begin{equation}\label{Emf}
\langle E_m(f) , \varphi\rangle_{V_m(\Omega)}
= \langle m f,\varphi\rangle_{L^2(\Omega)} \quad\forall\, \varphi\in V_m(\Omega).
\end{equation}
$E_m$ is clearly linear. Putting $\varphi=u$ in \eqref{volpe} and exploiting \eqref{stima3} and \eqref{c1},  we find
\begin{equation}\label{normauvm}
\|u\|_{V_m(\Omega)}\leq C\cdot C_1(m)\|m\|_{L^\infty(\Omega)}\|f\|_{L^2(\Omega)}.   
\end{equation} 
Therefore $E_m$ is a linear bounded operator such that
\begin{equation}\label{normaEm}
\|E_m\|_{\mathcal{L}(L^2(\Omega), V_m(\Omega))}\leq C\cdot C_1(m) \|m\|_{L^\infty(\Omega)}.
\end{equation} 
Let $i_m$ be the inclusion of $V_m(\Omega)$  into $L^2(\Omega)$. Note that, by compactness of the inclusion $H^1(\Omega) \hookrightarrow L^2(\Omega)  $ (see \cite{L}), it follows that $i_m$ is a compact operator as well. Moreover, 
we define a second the linear operator      
\begin{equation}\label{Gm}
G_m:V_m(\Omega)\to V_m(\Omega)
\end{equation} 
by $G_m=E_m\circ i_m$, i.e.  for all $f\in V_m(\Omega)$, $G_m(f)$ is defined by 
\begin{equation}\label{Gmf} 
\langle G_m(f) , \varphi\rangle_{V_m(\Omega)}
= \langle m f,\varphi\rangle_{L^2(\Omega)} \quad\forall\, \varphi\in V_m(\Omega).
\end{equation}        

The main properties of the operators $E_m$ and $G_m$ are summarized in the following 
Proposition.

\begin{proposition}\label{proprietaop} 
Let $m\in L^\infty(\Omega)$ such that $\int_\Omega m\, dx\neq 0$ and $P_m, E_m$ and $G_m$
defined by Definition \ref{def1}, \eqref{Emf} and \eqref{Gmf} respectively. Then\\
i) $E_m(f)= G_m( P_m(f))$ for all $f\in H^1(\Omega)$;\\
ii) $G_m$ is self-adjoint and compact;\\
iii) $E_m$ restricted to $H^1(\Omega)$ is compact.
\end{proposition}
\begin{proof}
 i) Let $f\in H^1(\Omega)$, then $P_m(f)\in V_m(\Omega)$. By \eqref{Gmf}, i) and iii) of Proposition 
 \ref{prop1} we have
 \begin{equation*}
\langle G_m(P_m(f)), \varphi\rangle_{V_m(\Omega)} 
= \langle mP_m(f), \varphi  
\rangle_{L^2(\Omega)}= 
\langle mf, P_m(\varphi)\rangle_{L^2(\Omega)}=\langle mf, \varphi\rangle_{L^2(\Omega)}\quad\forall\,\varphi\in 
V_m(\Omega).
\end{equation*}  
Thus, by \eqref{Emf},  $E_m(f)= G_m( P_m(f))$.\\   
ii)
For all $f, g\in V_m(\Omega)$, by \eqref{Gmf}, we have
\begin{equation*}
\langle G_m(f), g\rangle_{V_m(\Omega)} =
\langle m f, g\rangle_{L^2(\Omega)}
= \langle m g, f\rangle_{L^2(\Omega)}
=\langle G_m (g), f\rangle_{V_m(\Omega)},
 \end{equation*} 
then $G_m$ is self-adjoint. The compactness of the operator $G_m$  is an immediate consequence 
of its definition $G_m=E_m \circ i_m $, the inclusion $i_m$ being compact
and the operator $E_m$ continuous.\\
iii) It follows from i) and ii).
\end{proof} 
 
By the general theory of self-adjoint compact operators (see \cite{DF,Lax}),   
it follows that all nonzero eigenvalues of $G_m$ have a finite dimensional eigenspace 
and they can be obtained by the {\it Fischer's Principle}     
\begin{equation}\label{Fischer1} \mu_k(m) =  \sup_{F_k\subset V_m(\Omega)}
\inf_{f\in F_k\atop f\neq 0}
\cfrac{\langle G_m (f), f \rangle_{V_m(\Omega)} }{\|f\|^2_{V_m(\Omega)}} \,= \sup_{F_k\subset V_m(\Omega)}
\inf_{f\in F_k\atop f\neq 0}
\cfrac{\int_\Omega m f^2 \, dx }{\int_\Omega |\nabla f|^2\, dx} \,, \quad k=1, 2, 3, \ldots\end{equation}  
and
\begin{equation*}      
\mu_{-k}(m) =  \inf_{F_k\subset V_m(\Omega)}
\sup_{f\in F_k\atop f\neq 0}
\cfrac{\langle G_m (f), f \rangle_{V_m(\Omega)} }{\|f\|^2_{V_m(\Omega)}} \,=
\inf_{F_k\subset V_m(\Omega)}
\sup_{f\in F_k\atop f\neq 0}
\cfrac{\int_\Omega m f^2 \, dx }{\|f\|^2_{V_m(\Omega)}}\,,\quad k=1, 2, 3,  \ldots,
\end{equation*}  
where the first extrema are taken over all the subspaces $F_k$ of $V_m(\Omega)$ of dimension
$k$. As observed in \cite{DF}, all the inf's and sup's in the above characterizations of the eigenvalues are actually
assumed. Hence, they could be replaced by min's and max's and the eigenvalues are obtained 
exactly in correspondence of the associated eigenfunctions.
The sequence $\{\mu_k(m)\}$ contains all the real positive eigenvalues (repeated with their multiplicity), is decreasing and converging to zero, whereas $\{\mu_{-k}(m)\}$ is formed by all the real negative
eigenvalues (repeated with their multiplicity), is increasing and converging to zero.\\ 

We will write $\{m>0\}$ as a short form of $\{x\in \Omega: m(x)>0\}$ and similarly 
$\{m<0\}$ for $\{x\in \Omega: m(x)<0\}$. The following proposition
is analogous to \cite[Proposition 1.11]{DF}.

\begin{proposition}\label{segnorho}Let $m \in L^\infty(\Omega)$ and $G_m$ be the operator 
\eqref{Gmf}. 
Then, the following statements hold\\   
i) if $|\{m>0\}|=0$, then there are no positive eigenvalues;\\
ii) if $|\{m>0\}|>0$ and $\int_\Omega m\, dx<0$, then there is a sequence of positive 
eigenvalues $\mu_k(m)$;\\
iii) if $|\{m<0\}|=0$, then there are no negative eigenvalues;\\
iv) if $|\{m<0\}|>0$ and $\int_\Omega m\, dx>0$, then there is a sequence of negative 
eigenvalues $\mu_{-k}(m)$.
\end{proposition} 

\begin{proof}
i) Let $\mu$ be an eigenvalue and $u$ a corresponding eigenfunction. By \eqref{Gmf} with 
$f=\varphi=u$ we have    
$$\mu= \cfrac{\int_\Omega m u^2\, dx }{\|u\|^2_{V_m(\Omega)}}\, \leq 0.$$
ii) By measure theory covering theorems, for each positive integer $k$ there exist $k$ 
disjoint closed balls $B_1, \ldots, B_k$ in $\Omega$ such that $| B_i \cap \{m>0\}|>0$ 
for $i=1, \ldots, k$. Let $f_i\in C^\infty_0(B_i)$ such that 
$\int_\Omega m f_i^2 \, dx=1$ for every $i=1, \ldots, k$.
Note that the functions  $f_i$ are linearly independent. 
We put $g_i=P_m(f_i)$, where $P_m$ is defined in Definition \ref{def1};
$g_i\in V_m(\Omega)$ for all $i=1, \ldots, k$. 
We show that the functions $g_i$ are linearly independent as well.
Let $\alpha_1, \ldots , \alpha_k$ be constants such that
$\sum_{i=1}^k \alpha_i g_i=0$; this implies $P_m\left(\sum_{i=1}^k \alpha_i f_i\right)=0$,
i.e., by ii) of Proposition \ref{prop1}, $\sum_{i=1}^k \alpha_i f_i=c\in\mathbb{R}$.
Evaluating $\sum_{i=1}^k \alpha_i f_i$ in $\Omega\smallsetminus \cup_{i=1}^k B_i\neq 
\emptyset$ we find  $c=0$ and,  therefore,  $\alpha_i=0$ for  all $i=1, \ldots, k$.  
 
Let $F_k= \ $span$  \{g_1, \ldots, g_k\}$. $F_k$ is  
a subspace of $V_m(\Omega)$ of dimension $k$. For every $g\in F_k\smallsetminus 
\{0\}$, $g=\sum_{i=1}^k a_i g_i$, with suitable constants $a_i\in \mathbb{R}$. 
Let us put $f=\sum_{i=1}^k a_i f_i$,  clearly $g=P_m(f)$.
Then, by i) and iii) of Proposition \ref{prop1} and recalling that $\int_\Omega m\, dx<0$,  we have   
\begin{equation*} \begin{split}\cfrac{\langle G_m (g), 
g\rangle_{V_m(\Omega)}}{\|g\|^2_{V_m(\Omega)}}\, = &
\cfrac{\langle mg, g\rangle_{L^2(\Omega)}}{\|g\|^2_{V_m(\Omega)}} \, =
\cfrac{\langle mP_m(f), P_m(f)\rangle_{L^2(\Omega)}}{\|g\|^2_{V_m(\Omega)}} \, =
\cfrac{\langle mf, P_m(f)\rangle_{L^2(\Omega)}}{\|g\|^2_{V_m(\Omega)}} \, \\=& 
\cfrac{\langle mf, f\rangle_{L^2(\Omega)}-\left\langle mf, \int_\Omega mf\,dx /
\int_\Omega m\,dx\right\rangle_{L^2(\Omega)}}{\|g\|   
^2_{V_m(\Omega)}} \, \\=& 
\cfrac{\langle mf, f\rangle_{L^2(\Omega)}-\left(\int_\Omega mf\,dx\right)^2/
\int_\Omega m\,dx}{\|g\|^2_{V_m(\Omega)}} \,\\  \geq &  
\cfrac{\sum_{i, j=1}^k a_i a_j\int_\Omega m f_i f_j  \,dx}{\sum_{i,j=1}^k\langle g_i, g_j\rangle_{V_m(\Omega)} a_i a_j} \,   
=\cfrac{\sum_{i=1}^k a_i^2}{\sum_{i,j=1}^k\langle g_i, g_j\rangle_{V_m(\Omega)} a_i a_j} \,=
\cfrac{\|a\|^2_{\mathbb{R}^k}}{\langle A_k a, a\rangle_{\mathbb{R}^k}}\,\geq \cfrac{1}{\|A_k\|}\, >0,
\end{split}
\end{equation*}
where $\|a\|_{\mathbb{R}^k}$, $\|A_k\|$ and $\langle A_k a, a\rangle_{\mathbb{R}^k}$ denote, respectively,
the euclidean norm of the vector $a=(a_1, \ldots, a_k)$, the norm of the non null matrix 
$A_k=\left( \langle g_i, g_j\rangle_{V_m(\Omega)} \right)_{i, j=1}^k$ and the inner product in $\mathbb{R}^k$.
From the Fischer's Principle \eqref{Fischer1} we conclude that $\mu_k(m)\geq \cfrac{1}{\|A_k\|}\, >0$ for
every $k$.\\
The cases iii) and iv) are similarly proved.
\end{proof}

\subsection{Weak formulation of problem \eqref{p0}} 
The operators $E_m$ and $G_m$ are related to the following problem
with Neumann boundary conditions   
\begin{equation}\label{BVPm}
\begin{cases}-\Delta u = mf \quad &\text{in } \Omega\\
\cfrac{\partial u}{\partial \nu}=0 &\text{on } \partial \Omega. 
\end{cases}  
\end{equation}

For $m\in L^{\infty}(\Omega)$ and $f\in L^2(\Omega)$,  a \emph{weak solution} of problem \eqref{BVPm} is a function $u\in H^1(\Omega)$ such that
\begin{equation*}\label{FD1r}
\int_\Omega \nabla u\cdot \nabla\varphi \, dx =  \int_\Omega m f\varphi\, dx \quad\forall\, \varphi\in H^1(\Omega)  
\end{equation*}
or, equivalently,    
\begin{equation}\label{formadebh1}
\left\langle \nabla u,\nabla\varphi\right\rangle_{L^2(\Omega)}
=\langle m f,\varphi\rangle_{L^2(\Omega)} \quad\forall\, \varphi\in H^1(\Omega).
\end{equation}
The assumptions under which problem \eqref{BVPm} admits solutions are well known (see for
example \cite{Mik}). 
By using the tools introduced in Subsection \ref{operators}, we find those conditions independently.  
      
\begin{lemma}\label{equivalenza}
 Let $m \in L^\infty(\Omega)$ such that $\int_\Omega m\, dx \neq 0$ and
 $f\in L^2(\Omega)$. Then $u\in H^1(\Omega)$ satisfies
 \begin{equation}\label{formaequivpm} 
\left\langle \nabla u,\nabla\varphi\right\rangle_{L^2(\Omega)}
=\langle m P_m(f),\varphi\rangle_{L^2(\Omega)} \quad\forall\, \varphi\in H^1(\Omega)  
\end{equation} 
 if and only if $P_m(u)$ satisfies
 \begin{equation}\label{formaequiv} 
\left\langle P_m(u),\varphi\right\rangle_{V_m(\Omega)}
=\langle m f,\varphi\rangle_{L^2(\Omega)} \quad\forall\, \varphi\in V_m(\Omega).
\end{equation}
\end{lemma}
\begin{proof}
If $u\in H^1(\Omega)$ satisfies \eqref{formaequivpm}, for all $\varphi\in V_m(\Omega)$, by iv), i) and iii) of
Proposition \ref{prop1}, we have  
 \begin{equation*} 
\left\langle P_m(u),\varphi\right\rangle_{V_m(\Omega)}
=\left\langle\nabla P_m(u),\nabla\varphi\right\rangle_{L^2(\Omega)}
=\left\langle\nabla u,\nabla\varphi\right\rangle_{L^2(\Omega)}
=\langle m P_m(f),\varphi\rangle_{L^2(\Omega)}=\langle mf,\varphi\rangle_{L^2(\Omega)}.        
\end{equation*}      
Vice versa, let $u$ verify \eqref{formaequiv}. For all $\varphi\in H^1(\Omega)$,
recalling iv) and i) of Proposition \ref{prop1} we have
 \begin{equation*} 
 \begin{split} 
\left\langle \nabla u,\nabla\varphi\right\rangle_{L^2(\Omega)}
&=\left\langle\nabla P_m(u),\nabla P_m(\varphi)\right\rangle_{L^2(\Omega)}
=\left\langle P_m(u), P_m(\varphi)\right\rangle_{V_m(\Omega)}\\
&=\langle m f,P_m(\varphi)\rangle_{L^2(\Omega)} 
=\langle m P_m(f),\varphi\rangle_{L^2(\Omega)}.  
\end{split}    
\end{equation*}    
\end{proof}    

As the following proposition says, the operator $E_m$ provides the solutions of problem 
\eqref{BVPm}. 
 
\begin{proposition}
Let $m \in L^\infty(\Omega)$ such that $\int_\Omega m\, dx \neq 0$ and $E_m$ be the operator
\eqref{Em}. Then\\     
i) \eqref{formadebh1} has a solution if and only if $f\in L^2_m(\Omega)$;\\
ii) if $f\in L^2_m(\Omega)$, \eqref{formadebh1} has a unique solution $\overline{u}\in 
V_m(\Omega)$ and any other solution is of form   
$\overline{u} + c$, $c\in\mathbb{R}$;\\ 
iii) $\overline{u}=E_m(f)$, i.e. $\overline{u}$ is the unique solution of 
$$\langle \overline{u},\varphi\rangle_{V_m(\Omega)}=\langle m f,\varphi\rangle_{L^2(\Omega)}
\quad\forall\varphi\in V_m(\Omega);$$
iv) the estimate  \begin{equation*}
\|\overline{u}\|_{H^1(\Omega)}\leq C\cdot C_1(m)\left(C^2 \cdot C_1^2(m)+1\right)^{1/2} \|m\|_{L^\infty(\Omega)}\|f\|_{L^2(\Omega)}          
\end{equation*}      
holds.    
\end{proposition}
\begin{proof}
If \eqref{formadebh1} admits a solution $u$, choosing $\varphi\equiv 1$ we obtain $f\in L^2_m(\Omega)$.\\   
Vice versa, let $f\in L^2_m(\Omega)$. By Lemma \ref{equivalenza}, $u\in H^1(\Omega)$ is
a solution of  \eqref{formadebh1} if and only if $P_m(u)$ is a solution of \eqref{formaequiv}.
By \eqref{volpe} and \eqref{Emf}, we know that \eqref{formaequiv} admits the unique solution 
$\overline{u}=E_m(f)\in V_m(\Omega)$.         
Then the set of solutions of \eqref{formadebh1} is $\{u\in H^1(\Omega):P_m(u)=\overline{u}\}=\{u\in H^1(\Omega):u=\overline{u}+c,\, c\in\mathbb{R}\}$
and only $\overline{u}$ belongs to $V_m(\Omega)$.
This proves i), ii) and iii).\\
iv) It follows immediately from  \eqref{stimanorme2} and \eqref{normauvm}.
\end{proof}

Finally, we introduce the weak formulation of problem \eqref{p0}.  
A function $u\in H^1(\Omega)$ is said an \emph{eigenfunction} of \eqref{p0} associated to the \emph{eigenvalue} $\lambda$ if \begin{equation*}\label{FDautr}
\int_\Omega \nabla u\cdot \nabla\varphi \, dx = \lambda \int_\Omega m u\varphi\, dx \quad\forall\, \varphi\in H^1(\Omega)
\end{equation*}  

or, equivalently,  
\begin{equation}\label{falena}
\langle \nabla u,\nabla \varphi\rangle_{L^2(\Omega)}
=\lambda \langle m u,\varphi\rangle_{L^2(\Omega)} \quad\forall\, \varphi\in H^1(\Omega).
\end{equation} 
It is easy to check that zero is an eigenvalue and the associated eigenfunctions are all of the
constant functions.
\begin{proposition}\label{equivalenzalambdamu}
Let $m \in L^\infty(\Omega)$ such that $\int_\Omega m\, dx \neq 0$ and $G_m$ be
the operator \eqref{Gm}. 
Then the nonzero eigenvalues of problem \eqref{p0} are exactly the reciprocals 
of the nonzero eigenvalues of the operator $G_m$ and the correspondent eigenspaces coincide. 
\end{proposition}      
 \begin{proof}
If $\lambda\neq 0$ is an eigenvalue and $u$ is an associated eigenfunction of problem \eqref{p0},
choosing $\varphi\equiv 1$ in \eqref{falena}, we obtain $u\in V_m(\Omega)$. 
By \eqref{falena} and \eqref{prodscalvm} we have 
\begin{equation*}     
\left\langle \frac{u}{\lambda},\varphi\right\rangle_{V_m(\Omega)}
=\langle m u,\varphi\rangle_{L^2(\Omega)} \quad\forall\, \varphi\in V_m(\Omega)
\end{equation*}
and then, by definition \eqref{Gmf} of $G_m$, $G_m (u)= \cfrac{u}{\lambda}\, $. \\
Vice versa, let $G_m(u)=\mu u$, with $\mu\neq 0$. Then we have
\begin{equation*}
\left\langle \mu u,\varphi\right\rangle_{V_m(\Omega)}
=\langle m u,\varphi\rangle_{L^2(\Omega)} \quad\forall\, \varphi\in V_m(\Omega).
\end{equation*}
By iii) of Proposition \ref{prop1} we obtain 
\begin{equation*}
\left\langle P_m(\mu u),\varphi\right\rangle_{V_m(\Omega)}
=\langle m u,\varphi\rangle_{L^2(\Omega)} \quad\forall\, \varphi\in V_m(\Omega),
\end{equation*}
using Lemma \ref{equivalenza} we find       
\begin{equation*}    
\mu\left\langle\nabla u,\nabla\varphi\right\rangle_{L^2(\Omega)} 
=\langle m P_m(u),\varphi\rangle_{L^2(\Omega)} \quad\forall\, \varphi\in H^1(\Omega)  
\end{equation*} 
and finally, applying iii) of Proposition \ref{prop1} again, we conclude 
\begin{equation*}
\left\langle\nabla u,\nabla\varphi\right\rangle_{L^2(\Omega)}
=\frac{1}{\mu}\langle m u,\varphi\rangle_{L^2(\Omega)} \quad\forall\, \varphi\in H^1(\Omega),
\end{equation*}  
i.e. $1/\mu$ is an eigenvalue of \eqref{p0}.
\end{proof}     
  
Consequently, in general,
the eigenvalues of problem \eqref{p0} form two monotone sequences
$$ 0<\lambda_1(m)\leq \lambda_2(m)\leq\ldots\leq  \lambda_k(m)\leq \ldots$$
and
$$ \ldots\leq\lambda_{-k}(m)\leq\ldots\leq \lambda_{-2}(m)\leq  \lambda_{-1}(m)<0  ,$$
where every eigenvalue appears as many times as its multiplicity, the latter being finite
owing to the compactness of $G_m$. \\
The variational characterization \eqref{Fischer1} for $k=1$, assuming that $|\{m>0\}|>0$
and $\int_\Omega m\,dx<0$,  becomes  
  
\begin{equation}\label{mu1} \mu_1(m) =  \max_{f\in V_m(\Omega)\atop f\neq 0}
\cfrac{\langle G_m (f), f \rangle_{V_m(\Omega)} }{\|f\|^2_{V_m(\Omega)}} \,=\max_{f\in V_m(\Omega)\atop f\neq 0}
\cfrac{\int_\Omega m f^2 \, dx }{\int_\Omega |\nabla f|^2\, dx}.  
\end{equation}    
The maximum in \eqref{mu1} 
is obtained if and only if $f$ is an eigenfunction relative to $\mu_1$. 
Similarly, for $\lambda_1(m)$ we have
\begin{equation}\label{2a}\lambda_1(m) = \min_{u\in V_m(\Omega)\atop \int_\Omega m u^2dx>0}
\cfrac{\int_\Omega |\nabla u|^2\, dx}{\int_\Omega m u^2 \, dx}\,\end{equation}     
and the minimum in \eqref{2a} 
is obtained if and only if $u$ is an eigenfunction relative to $\lambda_1$.

We note that the characterization
\begin{equation*}
\quad\lambda_1(m) = \min_{u\in H^1(\Omega)\atop \int_\Omega m u^2dx>0}
\cfrac{\int_\Omega |\nabla u|^2\, dx}{\int_\Omega m u^2 \, dx}\,
\end{equation*}
also holds and it is more often used in the literature.  
  
\begin{proposition} Let $m\in L^\infty(\Omega)$ such that $|\{m>0\}|>0$ and $\int_\Omega m\, dx<0$. Then $\mu_1(m)$ is simple and any associated eigenfunction is one signed in $\Omega$.
\end{proposition}  

\begin{proof} Let $u\in V_m(\Omega)$ be an eigenfunction related to $\mu_1(m)$. Let us show 
that $|u|$ is an eigenfunction as well. Consider the  projection $P_m(|u|)$ of $|u|$ on $V_m(\Omega)$, where $P_m$ is defined in Definition \ref{def1}.
By \eqref{mu1} and  iv) of Proposition \ref{prop1}  we have 
\begin{equation*}   
\mu_1(m)\geq \cfrac{\int_\Omega m (P_m(|u|))^2 \, dx }{\int_\Omega |\nabla P_m(|u|)|^2\, dx}=
\cfrac{\int_\Omega m u^2 \, dx-\frac{\left(\int_\Omega m |u|\, dx\right)^2}{\int_\Omega m\,dx}}{\int_\Omega |\nabla u|^2\, dx}
\geq\cfrac{\int_\Omega m u^2 \, dx }{\int_\Omega |\nabla u|^2\, dx}=\mu_1(m). 
\end{equation*}
Therefore, we have the equality sign in the previous chain. In particular, we find $\int_\Omega m |u| 
\, dx=0$, then $|u|$ belongs to $V_m(\Omega)$ and finally, by iii) of Proposition \ref{prop1}, $|u|
=P_m(|u|)$ is an eigenfunction. 
By Proposition  \ref{equivalenzalambdamu}, $|u|$ satisfies the equation $-\Delta |u|
=\mu_1(m)^{-1} m|u|$ and, by Harnack inequality        
(see \cite{GT}), we conclude that $|u|>0$ in $\Omega$; therefore $u$ is one signed in $\Omega$.\
Let $u,v$ be two  eigenfunctions; set $\alpha=\frac{\int_\Omega v\,dx}{\int_\Omega u\,dx}$, note that $\int_\Omega (\alpha u- v)\, dx=0$. Note that also $\alpha u-v$ is an eigenfunction  of $\mu_1(m)$. If $\alpha u-v$ was not identically zero, then, it would be 
one signed and hence $\int_\Omega (\alpha u- v)\, dx\neq0$, reaching a contradiction. Therefore $v=\alpha u$ and $\mu_1(m) $ is simple.  
\end{proof}        

As a consequence of Proposition \ref{equivalenzalambdamu}, we have the following 
  
\begin{corollary} Let $m\in L^\infty(\Omega)$ such that $|\{m>0\}|>0$ and $\int_\Omega m\,     dx<0$. Then $\lambda_1(m)$ is simple and any associated eigenfunction is one signed in $\Omega$.    
\end{corollary}   
 We call $\lambda_1(m)$ the principal eigenvalue of problem \eqref{p0}.    
Throughout the paper we will denote by $u_m$ the unique positive eigenfunction of both $G_m$
(relative to $\mu_1(m)$)       
and problem \eqref{p0} (relative to $\lambda_1(m)$), normalized by 
\begin{equation}\label{normaliz1}
\|u_m\|_{V_m(\Omega)}=1,
\end{equation}
which is equivalent to 
\begin{equation}\label{normaliz2}  
\int_\Omega m u_m^2 \, dx= \mu_1(m)=\cfrac{1}{\lambda_1(m)}\,.
\end{equation}  
By standard regularity theory (see \cite{GT}), $u_m\in W^{2,2}_{\rm loc}(\Omega)\cap C^{1,
\beta}(\Omega)$ for all $0<\beta<1$.  

As last comment, we observe that $\mu_1(m)$ is homogeneous of degree 1, i.e. 
\begin{equation}\label{homo}\mu_1(\alpha m)= \alpha \mu_1(m) \quad \forall\, \alpha>0.\end{equation}
This follows immediately from \eqref{mu1}.
\medskip

\section{Rearrangements of measurable functions}\label{rearrangements}
In this section we introduce the concept of
rearrangement of a measurable function and summarize some related results we will use in the next section. 
The idea of rearranging a function dates back to the book \cite{hardy52} 
of Hardy, Littlewood and P\'olya, since than many authors have investigated both extensions and 
applications of this notion. Here we relies on the results in \cite{Alvino,B,B89,day70,
Kaw,ryff67}.    

Let $\Omega$ be an open bounded set of $\mathbb{R}^N$.

\begin{definition}
For every measurable 
function $f:\Omega\to\mathbb{R}$ the function $d_f:\mathbb{R}\to [0,|\Omega|]$ defined by
$$d_f(t)=|\{x\in\Omega: f(x)>t\}|$$
is called \emph{distribution function of $f$}.
\end{definition}

The symbol $\mu_f$ is also used. It is easy to prove the
following properties of $d_f$.

\begin{proposition}
For each $f$ the distribution function $d_f$ is decreasing, right continuous and
the following identities hold true
$$\lim_{t\to-\infty}d_f(t)=|\Omega|,\quad\quad \lim_{t\to\infty}d_f(t)=0.$$
\end{proposition}

\begin{definition}
Two measurable functions $f,g:\Omega \to \mathbb{R}$ are called \emph{equimeasurable} functions 
or \emph{rearrengements} of one another if one of the following equivalent conditions is satisfied

i)   $|\{x\in \Omega: f(x)>t\}|=|\{x\in \Omega: g(x)>t\}| \quad \forall\, t\in \mathbb{R}$;

ii)  $d_f=d_g$.
\end{definition}

Equimeasurability of $f$ and $g$ is denoted by $f\sim g$. Equimeasurable functions
share global extrema and integrals as it is stated precisely by the following proposition.
\begin{proposition}\label{rospo}
Suppose $f\sim g$ and let $F:\mathbb{R}\to\mathbb{R}$ be a Borel measurable function, then
 
i) $|f|\sim |g|$;

ii) $\esssup f=\esssup g$ and  $\essinf f=\essinf g$;

iii) $F\circ f\sim F\circ g$;

iv) $F\circ f\in L^1(\Omega)$ implies $F\circ g\in L^1(\Omega)$ and $\int_\Omega F\circ f\,dx=
\int_\Omega F\circ g\,dx$.

\end{proposition}
For a proof see, for example, \cite[Proposition 3.3]{day70} or \cite[Lemma 2.1]{B89}.

In particular, for each $1\leq p\leq\infty$, if $f\in L^p(\Omega)$ and $f\sim g$ then
$g\in L^p(\Omega)$ and 
\begin{equation*}\label{equi}
 \|f\|_{L^p(\Omega)}=\|g\|_{L^p(\Omega)}.
\end{equation*}

\begin{definition}
For every measurable function $f:\Omega\to\mathbb{R}$ the function $f^*:(0,|\Omega|)\to
\mathbb{R}$ defined by
$$f^*(s)=\sup\{t\in\mathbb{R}: d_f(t)>s\}$$
is called \emph{decreasing rearrangement of $f$}.
\end{definition}
An equivalent definition (used by some authors) is $f^*(s)=\inf\{t\in\mathbb{R}: d_f(t)\leq
s\}$. 

\begin{proposition} \label{furbi}
For each $f$ its decreasing rearrangement $f^*$ is decreasing, right continuous and
we have 
$$\lim_{s\to 0}f^*(s)=\esssup f\quad\text{and}\quad\lim_{s\to|\Omega|}f^*(s)=\essinf f.$$
Moreover, if $F:\mathbb{R}\to\mathbb{R}$ is a Borel measurable function then
$F\circ f\in L^1(\Omega)$ implies $F\circ f^*\in L^1(0,|\Omega|)$ and 
$$\int_\Omega F\circ f\,dx=\int_0^{|\Omega|} F\circ f^*\,ds.$$ 
Finally, $d_{f^*}=d_f$ and, for each measurable function $g$ we have $f\sim g$ if and only if $f^*=g^*$.
\end{proposition}

Some of the previous claims are simple consequences of the definition of $f^*$, for
more details see \cite[Chapter 2]{day70}. 

As before, it follows that, for each $1\leq p\leq\infty$, if $f\in L^p(\Omega)$ then
$f^*\in L^p(0,|\Omega|)$ and $\|f\|_{L^p(\Omega)}=\|f^*\|_{L^p(0,|\Omega|)}.$
 
\begin{definition}\label{prec1}
Given two functions $f,g\in L^1(\Omega)$, we write $g\prec f$ if
$$\int_0^t g^*\,ds\leq \int_0^t f^*\,ds\quad\forall\; 0\leq t\leq|\Omega|\quad\quad
{and}\quad\quad\int_0^{|\Omega|} g^*\,ds= \int_0^{|\Omega|} f^*\,ds.$$ 
\end{definition}
Note that $g\sim f$ if and only if $g\prec f$ and $f\prec g$. Among many properties of the relation
$\prec$ we mention the following (a proof is in \cite[Lemma 8.2]{day70}).

\begin{proposition}\label{prec}
For any pair of functions $f,g\in L^1(\Omega)$ and real numbers $\alpha$ and $\beta$,
if $\alpha\leq f\leq\beta$ a.e. in $\Omega$ and $g\prec f$ then $\alpha\leq g\leq\beta$ a.e. in $\Omega$.
\end{proposition}

\begin{proposition}\label{prec2}
For $f\in L^1(\Omega)$ let $g=\frac{1}{|\Omega|}\int_\Omega f \, dx$. Then we have 
$g\prec f$.
\end{proposition}

\begin{definition}\label{class} 
Let $f:\Omega\to\mathbb{R}$ a measurable function. We call the set
$$\mathcal{G}(f)=\{g:\Omega\to\mathbb{R}: g \text{ is measurable and } g\sim f \}$$
\emph{class of rearrangements of $f$} or \emph{set of rearrangements of $f$}.
\end{definition}

Note that, for $1\leq p\leq\infty$, if $f$ is in $L^p(\Omega)$ then $\mathcal{G}(f)$ is
contained in $L^p(\Omega)$.

As we will see in the next section, we are interested in the optimization of a functional 
defined on a class of rearrangements $\mathcal{G}(m_0)$,
where $m_0$ belongs to $L^\infty(\Omega)$. For this reason, although almost all of what follows
holds in a much more general context, hereafter we restrict our attention to classes of rearrangements of
functions in $L^\infty(\Omega)$.
We need compactness properties of the set $\mathcal{G}(m_0)$, with a little
effort it can be showed that this set is closed but in general it is not compact in the norm topology
of $L^\infty(\Omega)$. Therefore we focus our attention on the weak* compactness.
By $\overline{\mathcal{G}(m_0)}$ we denote the closure of $\mathcal{G}(m_0)$ in the weak* topology
of $L^\infty(\Omega)$.

\begin{proposition}\label{cane}
 Let $m_0$ be a function of $L^\infty(\Omega)$. Then $\overline{\mathcal{G}(m_0)}$ is
 
 i)  weakly* compact;
 
 ii) metrizable in the weak* topology;
 
 iii) sequentially weakly* compact. 
\end{proposition}
For the proof see \cite[Proposition 3.6]{ACF}. 

Moreover, the sets $\mathcal{G}(m_0)$ and $\overline{\mathcal{G}(m_0)}$ have further
properties. 

\begin{definition}\label{libellula}
Let $C$ be a convex set of a real vector space. An element $v$ in $C$ is said an \emph{extreme point of
$C$} if for every $u$ and $w$ in $C$ the identity $v=\frac{u+w}{2}$ implies $u=w$.
\end{definition}
A vertex of a convex polygon is an example of extreme point.

\begin{proposition}\label{convexity}
Let $m_0$ be a function of $L^\infty(\Omega)$, then

i) $\overline{\mathcal{G}(m_0)}=\{f\in L^\infty(\Omega): f\prec m_0\}$, 

ii) $\overline{\mathcal{G}(m_0)}$ is convex,

iii) $\mathcal{G}(m_0)$ is the set of the extreme points of $\overline{\mathcal{G}(m_0)}$.
\end{proposition}
\begin{proof}
The claims follow from \cite[Theorems 22.13, 22.2, 17.4, 20.3]{day70}.  
\end{proof}

An evident consequence of the previous theorem is that $\overline{\mathcal{G}(m_0)}$ is
the weakly* closed convex hull of $\mathcal{G}(m_0)$.
\begin{corollary}\label{minore}
 Let $m_0\in L^\infty(\Omega)$ and $m,q\in\overline{\mathcal{G}(m_0)}$. Then\\
 i) $\int_\Omega m\,dx=\int_\Omega m_0\,dx$;\\
 ii) assuming $\int_\Omega m\,dx\neq 0$, $m=q$ if and only if $m$ and $q$ are linearly dependent.
\end{corollary}        
\begin{proof}
i) It folllows immediately by i) Proposition \ref{convexity}, Definition \ref{prec1} and Proposition 
 \ref{furbi}.\\
 ii) If $m$ and $q$ are linearly dependent, then, without loss of generality we can assume that
 $m=\alpha q$, for some $\alpha\in\mathbb{R}$. Integrating over $\Omega$ and using i)
 we find $m=q$.        
\end{proof}  
 
The following is \cite[Theorem 11.1]{day70} rephrased for our case.
\begin{proposition}
Let $u\in L^1(\Omega)$ and $m_0\in L^\infty(\Omega)$. Then 
\begin{equation}\label{day}
\int_0^{|\Omega|} m_0^*(|\Omega|-s) u^*(s)\, ds\leq\int_\Omega m\,u\, dx
\leq\int_0^{|\Omega|} m_0^*(s) u^*(s)\, ds\quad\quad\forall m\in\mathcal{G}(m_0),
\end{equation}
moreover both sides of \eqref{day} are taken on.
\end{proposition}

The previous proposition implies that the linear optimization problems
\begin{equation}\label{max}
 \sup_{m \in \mathcal{G}(m_0)} \int_\Omega m u \, dx\end{equation}
 and
 \begin{equation*}
 \inf_{m \in \mathcal{G}(m_0)} \int_\Omega m u \, dx\end{equation*}
 admit solution.

Finally, we recall the following result proved in \cite[Theorem 5]{B}.

\begin{proposition}\label{Teobart}
Let $u\in L^1(\Omega)$ and $m_0\in L^\infty(\Omega)$. If problem \eqref{max} has a unique solution 
$m_M$, 
then there exists an increasing function $\psi$ such that $m_M=\psi \circ u$ a.e. in 
$\Omega$.
\end{proposition}

\section{Qualitative properties of $\mu_k(m)$  and optimization of $\mu_1(m)$}
\label{existence}   

In this section we will prove some qualitative properties of the eigenvalues $\mu_{k}(m)$, 
$k=1,2,3,\ldots$,  of the operator $G_m$ defined in \eqref{Gmf}, interesting on their 
own; then, we will use them to investigate the optimization of $\mu_{1}(m)$.
By Proposition \ref{equivalenzalambdamu}, as an immediate consequence, we obtain the 
corresponding result for $\lambda_{1}(m)$. 
  
We introduce the following convex subset of $L^{\infty}(\Omega)$ 
\begin{equation*}
L^\infty_<(\Omega)=\left\{m\in L^\infty (\Omega): \int_\Omega m\,dx <0\right\}. \end{equation*}

Observe that, by Proposition \ref{segnorho}, 
$\mu_k(m)$ and $u_m$ (the unique positive eigenfunction of $\mu_1(m)$ of problem \eqref{p0}
normalized as in \eqref{normaliz1}) are well defined only when $|\{m>0\}|>0$. We extend 
them
to the whole space $L^\infty_<(\Omega)$ by putting   
\begin{equation} \label{muk}  
\widetilde{\mu}_k(m)=\begin{cases} \mu_k(m) \quad & \text{if } |\{m>0\}|>0\\
0 & \text{if } |\{m>0\}|=0\end{cases}
\end{equation}
and
\begin{equation}  \label{um}   
\widetilde{u}_m=\begin{cases}u_m \quad & \text{if } |\{m>0\}|>0\\
0 & \text{if } |\{m>0\}|=0.\end{cases}
\end{equation}     

\begin{remark}\label{oss1}
Note that $\widetilde{\mu}_k(m)=0$ if and only if $|\{m>0\}|=0$ a.e. in $\Omega$ and, in this 
circumstance,  the inequality
\begin{equation}\label{tilde}\sup_{F_k\subset V_m(\Omega)}
\min_{f\in F_k\atop f\neq 0}
\cfrac{\langle G_m f, f \rangle_{V_m(\Omega)} }{\|f\|^2_{V_m(\Omega)}} \,\leq 0 
\end{equation}
holds, where $F_k$ varies among all the $k$-dimensional subspaces of $V_m(\Omega)$. 
Moreover,  from \eqref{homo}, we have $\widetilde{\mu}_1(\alpha m)=\alpha \widetilde{\mu} 
_1(m)$ for every $\alpha\geq 0$.
\end{remark}

\begin{lemma}\label{teo1} 
Let $m\in L_<^\infty(\Omega)$,  $E_m$ be the linear operator \eqref{Em}. Then, the map $m\mapsto E_m$ is sequentially weakly* continuous from $L_<^\infty(\Omega)$
to $\mathcal{L}(H^1(\Omega),H^1(\Omega)) $ endowed with the norm topology.
\end{lemma} 
  
\begin{proof}

i) Let $\{m_i\}$ be a sequence which weakly* converges to $m$ in $L^\infty_<(\Omega)$. 
Being $\{m_i\}$ bounded in $L^\infty(\Omega)$, there exists a constant $M>0$ such that
\begin{equation}\label{M}  
\|m\|_{L^\infty(\Omega)}\leq M\quad \text{ and } \quad \|m_i\|_{L^\infty(\Omega)}\leq M
\quad \forall\, i. 
\end{equation} 
We begin by proving that            
$E_{m_i}(f)$ tends to $E_m(f)$ in $H^1(\Omega)$ for any fixed $f\in H^1(\Omega)$.
Recalling i) of Proposition \ref{proprietaop}, we put $u_i=E_{m_i}(f)= G_{m_i}(P_{m_i}(f))$ and 
$u=E_{m}(f)= G_{m}(P_{m}(f))$.\\ 
First, we show that $P_m(u_i)$ weakly converges to $u$ in $V_m(\Omega)$; indeed, by \eqref{Emf} we have             
\begin{equation*}
\langle u_i,\varphi\rangle_{V_{m_i}(\Omega)}=
\langle m_if,\varphi\rangle_{L^2(\Omega)}\quad \forall\varphi \in V_{m_i}(\Omega).
\end{equation*}
By Lemma \ref{equivalenza} and iii) of Proposition \ref{prop1}, we find  
\begin{equation}\label{ui}
\langle\nabla u_i,\nabla\varphi\rangle_{L^2(\Omega)}= \langle m_i P_{m_i}(f), 
\varphi\rangle_{L^2(\Omega)} \quad\forall\, \varphi\in
 H^1(\Omega).
\end{equation}
Similarly, for $u$ we have
\begin{equation}\label{u}  
\langle\nabla u,\nabla\varphi\rangle_{L^2(\Omega)}= \langle m P_{m}(f),\varphi\rangle_{L^2(\Omega)} \quad\forall\, \varphi\in H^1(\Omega).
\end{equation}
Taking $\varphi\in V_m(\Omega)$ in \eqref{ui} and in \eqref{u} we
find 
\begin{equation}\label{A01}
\langle u_i,\varphi\rangle_{V_m(\Omega)}=
\langle m_iP_{m_i}(f),\varphi\rangle_{L^2(\Omega)}\quad \forall\varphi \in V_m(\Omega)
\end{equation}
and
\begin{equation}\label{A02}
\langle u,\varphi\rangle_{V_m(\Omega)}=
\langle mP_{m}(f),\varphi\rangle_{L^2(\Omega)}\quad \forall\varphi \in V_m(\Omega).
\end{equation}
By using iii) of Proposition  \eqref{prop1}, equation \eqref{A01}  becomes
\begin{equation}\label{A03}      
\langle P_m(u_i),\varphi\rangle_{V_m(\Omega)}=
\langle m_iP_{m_i}(f),\varphi\rangle_{L^2(\Omega)}\quad \forall\varphi \in V_m(\Omega)
\end{equation}                    
and,   subtracting \eqref{A02} from \eqref{A03}, we get        
\begin{equation}\label{sottrazione}
\langle P_m(u_i)-u,\varphi\rangle_{V_m(\Omega)}=
\langle m_iP_{m_i}(f)-m P_m(f),\varphi\rangle_{L^2(\Omega)}\quad \forall\varphi \in V_m(\Omega).
\end{equation}
As a consequence of the weak* convergence of $m_i$ to $m$ in $L^\infty(\Omega)$, letting $i\to
\infty$ we obtain $\int_\Omega m_i f\varphi \,dx \to \int_\Omega m f\varphi \,dx$, $\int_\Omega m_i  
f\,dx \to \int_\Omega m f\,dx$, $\int_\Omega m_i\varphi \,dx \to \int_\Omega m\varphi \,dx$ and $
\int_\Omega m_i\,dx \to \int_\Omega m\,dx$, which imply that  the right hand term goes to zero,  
thus $P_m(u_i)$ weakly converges to $u$ in
$V_m(\Omega)$.
By exploiting the continuity of the inclusion $V_m(\Omega)\hookrightarrow H^1(\Omega)$ and the 
compactness of $H^1(\Omega)\hookrightarrow L^2(\Omega)$, we deduce that $P_m(u_i)$ weakly 
converges to $u$ in $H^1(\Omega)$ and strongly in $L^2(\Omega)$.
Putting $\varphi=P_m(u_i)-u$ in \eqref{sottrazione} we get
\begin{equation}\label{inter}    
\begin{split} 
 \| P_m(u_i)-u\|^2_{V_m(\Omega)}&=
\langle m_iP_{m_i}(f)-m P_m(f),P_m(u_i)-u\rangle_{L^2(\Omega)}\\
&\leq\left(\| m_iP_{m_i}(f)\|_{L^2(\Omega))}+\| mP_{m}(f)\|_{L^2(\Omega)}\right)\| P_m(u_i)-u\|_{L^2(\Omega)}.
\end{split} 
\end{equation}
By using \eqref{M}, \eqref{stima1} and \eqref{c1}, we find
\begin{equation}\label{59}
  \| mP_{m}(f)\|_{L^2(\Omega)}\leq\,M\| P_{m}(f)\|_{L^2(\Omega)}
  \leq \, M C_1(m)\|f\|_{L^2(\Omega)}.  
\end{equation}     
Similarly,  we have
 \begin{equation}\label{55}
  \| m_iP_{m_i}(f)\|_{L^2(\Omega)}\leq M C_1(m_i)\|f\|_{L^2(\Omega)}.
\end{equation}       
By weak*  convergence of $m_i$ to $m$ we can assume 
\begin{equation*}
 \left|\int_\Omega m_i\,dx\right|> \frac{\left|\int_\Omega m\,dx\right|}{2}  
\end{equation*}
for $i$ large enough. Therefore 
\begin{equation}\label{stimami}  
C_1(m_i)\leq 1+ \frac{M}{\left|\int_\Omega m_i\, dx \right|}\, |\Omega|\leq 1+ \frac{2M}{\left|\int_\Omega m\, dx \right|}\, |\Omega| 
\end{equation}      
and, trivially
\begin{equation}\label{stimam} 
C_1(m)\leq 1+ \frac{2M}{\left|\int_\Omega m\, dx \right|}\, |\Omega|.  
\end{equation}   
For sake of simplicity we put
\begin{equation}\label{DmM}     
D(m,M)=1+ \frac{2M}{\left|\int_\Omega m\, dx \right|}\, |\Omega|.  
\end{equation}       
Then, by \eqref{inter}, \eqref{59}, \eqref{55}, \eqref{stimami} and \eqref{DmM} we find        
 \begin{equation*}                
 \| P_m(u_i)-u\|_{V_m(\Omega)}^2\leq MD(m,M)\|f\|_{L^2(\Omega)} \| P_m(u_i)-u\|_{L^2(\Omega)},         
\end{equation*}             
 from which the convergence of  $P_{m}(u_i)$ to $u$ in $V_m(\Omega)$ follows. 
 The next step shows that, actually, $u_i$ strongly converges to $u$ in $L^2(\Omega)$. Indeed, 
 by using Definition \ref{def1}, vi) of Proposition \ref{prop1}, \eqref{stima1},  \eqref{stimami}
 and \eqref{DmM}, we have             
\begin{equation*}
\begin{split}  
 \|u_i-u\|_{L^2(\Omega)}=&\left \| -\frac{\int_\Omega m_iu\,dx}{\int_\Omega m_i\, dx}+P_{m_i}\left(P_m(u_i)-u\right)\right\|_{L^2(\Omega)}\\
 \leq & \left \| \frac{\int_\Omega m_iu\,dx}{\int_\Omega m_i\, dx}\right\|_{L^2(\Omega)}+ \left \|P_{m_i}\left(P_m(u_i)-u\right)\right\|_{L^2(\Omega)}\\
 \leq  &\left | \frac{\int_\Omega m_iu\,dx}{\int_\Omega m_i\, dx}\right| |\Omega|^{1/2}+
 D(m,M)\left\|P_m(u_i)-u\right\|_{L^2(\Omega)},            
\end{split}            
 \end{equation*}
 which goes to zero because $\int_\Omega m_i  
u\,dx/\int_\Omega m_i \,dx \to \int_\Omega m  
u\,dx/\int_\Omega m \,dx=0$ ($u\in V_m(\Omega)$) and $P_m(u_i)\to u$ in $L^2(\Omega)$.     
 Moreover, by iv) of Proposition \ref{prop1}, we have    
 \begin{equation*}  
  \| u_i-u\|_{H^1(\Omega)}=\left(\|u_i-u\|^2_{L^2(\Omega)}+\| P_m(u_i)-u\|_{V_m(\Omega)}^2\right)^{1/2}  
 \end{equation*}
 and then $u_i$ converges to $u$ in $H^1(\Omega)$.   
 Summarizing, for every  $f\in H^1(\Omega)$ we have   
\begin{align*}
\|E_{m_i}(f) - E_m (f)\|_{H^1(\Omega)}\to 0\quad\text{for }i\to\infty.  
\end{align*}

Now, for fixed $i$, let $\{f_{i,j}\}$, $j=1,2, 3,\ldots$, be a maximizing sequence of
$$\sup_{g\in H^1(\Omega)\atop \|g\|_{H^1(\Omega)}\leq 1}
 \|E_{m_i}(g) - E_{m}(g)\|_{H^1(\Omega)}=\|E_{m_i}-E\|_{\mathcal{L}(H^1(\Omega), H^1(\Omega))} .$$      
Then, being $\|f_{i,j}\|_{H^1(\Omega)}\leq 1$, we can extract a subsequence (still denoted by $\{f_{i,j}\}$)
weakly convergent to some $f_i\in H^1(\Omega)$. Since  the operators
$E_{m_i}$ and $E_m$ restricted to $H^1(\Omega)$ are compact (see iii) of Proposition 
\ref{proprietaop}), it follows that
$E_{m_i}(f_{i,j})$ converges to $E_{m_i}(f_{i})$ and  
$E_{m}(f_{i,j})$ converges to $E_{m}(f_{i})$ strongly in $V_m(\Omega)$ and then in
$H^1(\Omega)$ as $j$ goes to
$\infty$. Thus we find
$$ \|E_{m_i}-E_m\|_{\mathcal{L}(H^1(\Omega), H^1(\Omega))}
=\|E_{m_i}(f_i) - E_m(f_i)\|_{H^1(\Omega)}.$$
This procedure yields a sequence $\{f_i\}$ in $H^1(\Omega)$ such that $\|f_{i}\|_{H^1(\Omega)}\leq 1$ 
for all $i$. Then, up to a subsequence, we can assume that
$\{f_i\}$ weakly converges  to a function $f\in H^1(\Omega)$ 
and (by compactness of the inclusion $H^1(\Omega)\hookrightarrow L^2(\Omega)$) strongly in $L^2(\Omega)$.  
By using \eqref{stimanorme2}, \eqref{stimami}, \eqref{stimam}, \eqref{DmM}, \eqref{normaEm} and \eqref{M} we find   
\begin{align*} & \|E_{m_i}- E_m\|_{\mathcal{L}(H^1(\Omega), H^1(\Omega))}=
\|E_{m_i}(f_i) - E_m(f_i)\|_{H^1(\Omega)}\\
& \leq \|E_{m_i}(f) - E_m(f)\|_{H^1(\Omega)} + \|E_{m_i}(f_i-f)\|_{H^1(\Omega)} +\|E_m(f_i-f)\|_
{H^1(\Omega)}\\
& \leq \|E_{m_i}(f) - E_m(f)\|_{H^1(\Omega)}
+ \left(C^2\cdot C_1^2(m_i)+1\right)^{1/2}\| 
E_{m_i}(f_i-f)\|_{V_{m_i}(\Omega)}\\
&+\left(C^2\cdot C_1^2(m)+1\right)^{1/2}\|E_m(f_i-f)\|_  
{V_{m}(\Omega)}\\       
&\leq \|E_{m_i}(f) - E_m(f)\|_{H^1(\Omega)}\\  
& +\left(C^2\cdot D(m,M)^2+1\right)^{1/2}\left(\|E_{m_i}\|  
_{\mathcal{L}(L^2(\Omega), V_{m_i}(\Omega))}+\|E_{m}\|_{\mathcal{L}(L^2(\Omega), 
V_m(\Omega))} \right)   
\|f_i-f\|_{L^2(\Omega)} \\      
&\leq \|E_{m_i}(f) - E_m(f)\|_{H^1(\Omega)}\\
& +\left(C^2\cdot D(m,M)^2+1\right)^{1/2} \left(C\cdot C_1(m_i)\|m_i\|_{L^\infty(\Omega)}
+C\cdot C_1(m)\|m\|_{L^\infty(\Omega)}      
 \right) \|f_i-f\|_{L^2(\Omega)}\\
 &\leq \|E_{m_i}(f) - E_m(f)\|_{H^1(\Omega)}+CMD(m,M)\left(C^2\cdot D(m,M)^2+1\right)^{1/2}  
 \|f_i-f\|_{L^2(\Omega)}.     
 \end{align*}    
Therefore $E_{m_i}$ converges to $E_m$ in the operator norm.
\end{proof}
\begin{remark}
We note that the previous lemma still holds replacing $L_<^\infty(\Omega)$ by  
the set of $L^\infty(\Omega)$ such that $\int_\Omega m\;dx\neq 0$.             
\end{remark}

\begin{lemma}\label{teo1ii}
Let $m\in L_<^\infty(\Omega)$, $\widetilde{\mu}_k(m)$       
as defined in \eqref{muk} for $k=1,2,3,\ldots$ and $\widetilde{u}_m$ as in \eqref{um}. Then\\
i) the map $m\mapsto \widetilde{\mu}_k(m)$ is sequentially weakly* continuous in
$L_<^\infty(\Omega)$; \\
ii) the map $m\mapsto \widetilde{\mu}_1(m)\widetilde{u}_m$ is sequentially weakly* continuous
from $L_<^{\infty}(\Omega)$ to $H^1(\Omega)$ (endowed with the norm topology). In particular, 
for any sequence $\{m_i\}$ weakly* convergent to $m\in L_<^\infty(\Omega)$, with $\widetilde{\mu}_1(m)>0$, then 
$\{\widetilde{u}_{m_i}\}$ converges to $\widetilde{u}_{m}$ in $H^1(\Omega)$.
\end{lemma}
\begin{proof}

i)   Let $\{m_i\}$ be a sequence which weakly* converges to $m$ in $L^\infty_<(\Omega)$. 
Being $\{m_i\}$ bounded in $L^\infty(\Omega)$, there exists a constant $M>0$ such that
\eqref{M} holds.
Let  $k=1, 2, 3, \ldots$, we show that  
\begin{equation}\label{modulo}
 |\widetilde{\mu}_k(m_i)- \widetilde{\mu}_k(m)|\leq D(m,M)(C^2\cdot D^2(m,M)+1)^{1/2} \|E_{m_i}-E_m\|_{\mathcal{L}(H^1(\Omega), H^1(\Omega))},   
\end{equation}  
where $D(m,M)$ is the constant in \eqref{DmM}. 
By Lemma \ref{teo1} the claim follows.   
We split the argument in three cases.\\  
 
Case 1.  $\widetilde{\mu}_k(m_i)$ ($i$ fixed), $\ \widetilde{\mu}_k(m)>0$.\\ Following       \cite[Theorem 2.3.1]{He} and by means of the Fischer's Principle \eqref{Fischer1} we have 
\begin{equation*} \begin{split}
\widetilde{\mu}_k(m_i)- \widetilde{\mu}_k(m) &=  \max_{F_k\subset V_{m_i}(\Omega)}
\min_{f\in F_k\atop f\neq 0}
\cfrac{\langle G_{m_i} f, f \rangle_{V_{m_i}(\Omega)} }{\|f\|^2_{V_{m_i}(\Omega)}} \,
- \max_{F_k\subset V_m(\Omega)}
\min_{f\in F_k\atop f\neq 0}
\cfrac{\langle G_m f, f \rangle_{ V_m(\Omega)} }{\|f\|^2_{ V_m(\Omega)}} \,\\
& \leq 
\min_{f\in \overline{F_k}\atop f\neq 0}
\cfrac{\langle G_{m_i} f, f \rangle_{V_{m_i}(\Omega)} }{\|f\|^2_{V_{m_i}(\Omega)}} \,
- 
\min_{f\in P_m (\overline{F_k})\atop f\neq 0}
\cfrac{\langle G_m f, f \rangle_{V_m(\Omega)} }{\|f\|^2_{V_m(\Omega)}} \,\\
& \leq 
\cfrac{\langle G_{m_i} \overline{f}, \overline{f} \rangle_{V_{m_i}(\Omega)} }{\|\overline{f}\|^2_{V_{m_i}(\Omega)}} \,-
\cfrac{\langle G_m (P_m(\overline{f})), P_m(\overline{f}) \rangle_{V_m(\Omega)} }{\|P_m(\overline{f})\|^2_{V_m(\Omega)}}\, ,  
\end{split} 
\end{equation*}
where $\overline{F_k}$ is a $k$-dimensional subspace of $V_{m_i}(\Omega)$ such that (note that, by vii) of Proposition \ref{prop1},  $P_m (\overline{F_k})$ is a $k$-dimensional subspace of $V_{m}(\Omega)$)             
$$\max_{F_k\subset V_{m_i}(\Omega)}
\min_{f\in F_k\atop f\neq 0}
\cfrac{\langle G_{m_i} f, f \rangle_{V_{m_i}(\Omega)} }{\|f\|^2_{V_{m_i}(\Omega)}} \, =\min_{f\in \overline{F_k}\atop f\neq 0}
\cfrac{\langle G_{m_i} f, f \rangle_{V_{m_i}(\Omega)} }{\|f\|^2_{V_{m_i}(\Omega)}} \,$$
and $\overline{f}$ is a function in $\overline{F_k}$ such that 
$$ \min_{f\in P_m(\overline{F_k})\atop f\neq 0}   
\cfrac{\langle G_m f, f \rangle_{V_m(\Omega)} }{\|f\|^2_{V_m(\Omega)}} \,=
\cfrac{\langle G_m(P_m(\overline{f})), P_m(\overline{f}) \rangle_{V_m(\Omega)} }
{\|P_m(\overline{f})\|^2_{V_m(\Omega)}} \,.
$$      
By iv) of Proposition \ref{prop1}  $\nabla P_m(\overline{f})=\nabla\overline{f}$ and $\nabla G_m(P_m(\overline{f}))=\nabla P_{m_i}(G_m(P_m(\overline{f})))$ hold, thus we have
\begin{equation*}   
\begin{split}
\widetilde{\mu}_k(m_i)- \widetilde{\mu}_k(m) & \leq 
\cfrac{\langle G_{m_i} \overline{f}, \overline{f} \rangle_{V_{m_i}(\Omega)} }{\|\overline{f}\|^2_{V_{m_i}(\Omega)}} \,-
\cfrac{\langle P_{m_i}(G_m (P_m(\overline{f}))),\overline{f}\rangle_{V_{m_i}(\Omega)} }{\|\overline{f}\|^2_{V_{m_i}(\Omega)}} \,\\
& = \cfrac{\langle (G_{m_i}-P_{m_i}\circ G_m\circ P_m)(\overline{f}), \overline{f} \rangle_{V_{m_i}(\Omega)} }{\|\overline{f}\|^2_{V_{m_i}(\Omega)}} \,\leq  
 \cfrac{\|(G_{m_i}-P_ {m_i}\circ G_m\circ P_m)(\overline{f})\|_{V_{m_i}(\Omega)}}{\|\overline{f}\|_{V_{m_i}(\Omega)}}\,.
\end{split}  
\end{equation*}
Then, using the identities 
$$G_{m_i}-P_{m_i}\circ G_m\circ P_m=P_{m_i}\circ G_{m_i}\circ P_{m_i}-P_{m_i}\circ G_m\circ P_m=P_{m_i}\circ(G_{m_i}\circ P_{m_i}-G_m\circ P_m),$$ 
\eqref{stimanorme2}, \eqref{stima2}, \eqref{stimami}, \eqref{c1}, \eqref{DmM} 
and i) of Proposition \ref{proprietaop} we find        
\begin{equation*}
\begin{split}     
\widetilde{\mu}_k(m_i)- \widetilde{\mu}_k(m) &\leq\cfrac{\| P_{m_i}\circ(G_{m_i}\circ P_{m_i}-G_m\circ P_m)(\overline{f})\|_{V_{m_i}(\Omega)}}{\|\overline{f}\|_{V_{m_i}(\Omega)}}\\
&=\cfrac{\| P_{m_i}\circ(G_{m_i}\circ P_{m_i}-G_m\circ P_m)(\overline{f})\|_{H^1(\Omega)}}{\|\overline{f}\|_{V_{m_i}(\Omega)}}\\                                 
&\leq (C^2\cdot C_1^2(m_i)+1)^{1/2}\|P_{m_i}\|_{\mathcal{L}(H^1(\Omega), H^1(\Omega))}\cfrac{\|(G_{m_i}\circ P_{m_i}-G_m\circ P_m)(\overline{f})\|_{H^1(\Omega)} }{\|\overline{f}\|_{H^1(\Omega)}}\\     
&\leq C_1(m_i)(C^2\cdot C_1^2(m_i)+1)^{1/2}\|G_{m_i}\circ P_{m_i}-G_m\circ P_m\|
_{\mathcal{L}(H^1(\Omega), H^1(\Omega))}\\       
&\leq D(m,M)(C^2\cdot D^2(m,M)+1)^{1/2} \|E_{m_i}-E_m\|_{\mathcal{L}(H^1(\Omega), H^1(\Omega))}.     
\end{split} 
\end{equation*}   
Interchanging the role of $m_i$ and $m$ and replacing \eqref{stimami} by \eqref{stimam}, we  
also have  
\begin{equation*}
 \widetilde{\mu}_k(m)- \widetilde{\mu}_k(m_i)\leq D(m,M)(C^2\cdot D^2(m,M)+1)^{1/2} \|E_{m_i}-E_m\|_{\mathcal{L}(H^1(\Omega), H^1(\Omega))}
\end{equation*} 
 and finally \eqref{modulo}. 

Case 2. $\widetilde{\mu}_k(m_i)>0$, $ \widetilde{\mu}_k(m)=0$ (and similarly in the case $\widetilde{\mu}_k(m)>0$, $ \widetilde{\mu}_k(m_i)=0$).\\
Note that in this case \eqref{tilde} holds for the weight function $m$. Then
the previous argument still applies provided that we replace the first step of the inequality chain by 
\begin{equation*} |\widetilde{\mu}_k(m_i)- \widetilde{\mu}_k(m) |= 
\widetilde{\mu}_k(m_i) \leq \max_{F_k\subset V_{m_i}(\Omega)}
\min_{f\in F_k\atop f\neq 0}
\cfrac{\langle G_{m_i} f, f \rangle_{V_{m_i}(\Omega)} }{\|f\|^2_{V_{m_i}(\Omega)}} \,
- \sup_{F_k\subset V_m(\Omega)}
\min_{f\in F_k\atop f\neq 0}
\cfrac{\langle G_m f, f \rangle_{V_m(\Omega)} }{\|f\|^2_{V_m(\Omega)}} \,.
\end{equation*}
  
Case 3. $\widetilde{\mu}_k(m_i)= \widetilde{\mu}_k(m)=0$.\\
In this case \eqref{modulo} is obvious.\\  
Therefore statement i) is proved.\\                
ii) Let $\{m_i\}, m$ be such that $m_i$ is weakly$^*$ convergent
to $m\in L_<^\infty(\Omega)$. By using \eqref{stimanorme2}, \eqref{stimami} and 
\eqref{DmM}, for $i$ sufficiently large, we have  $$\|\widetilde{u}_{m_i}\|_{H^1(\Omega)} 
\leq\left(C^2\cdot D(m,M)^2+1\right)^{1/2},$$           
up to a subsequence we can assume that $\widetilde{u}_{m_i}$ is weakly convergent to $z
\in H^1(\Omega)$, strongly in $L^2(\Omega)$ and pointwisely a.e. in $\Omega$. \\
First suppose $\widetilde{\mu}_1(m)=0$. Then, by i) $\widetilde{\mu}_1(m_i) \widetilde{u}_{m_i}$
weakly converges in $H^1(\Omega)$ to $\widetilde{\mu}_1(m)z=0= \widetilde{\mu}_1(m) \widetilde{u}_{m}$. Moreover, $\|\widetilde{\mu}_1(m_i) \widetilde{u}_{m_i}\|_{H^1(\Omega)}
=\widetilde{\mu}_1(m_i)\| \widetilde{u}_{m_i}\|_{H^1(\Omega)}$ tends to $0=\|\widetilde{\mu}_1(m) \widetilde{u}_{m}\|_{H^1(\Omega)}$. Therefore $\widetilde{\mu}_1(m_i) \widetilde{u}_{m_i}$ 
strongly converges to $\widetilde{\mu}_1(m) \widetilde{u}_{m}$ in $H^1(\Omega)$.\\
Next, consider the case $\widetilde{\mu}_1(m)>0$. By i) we have $\widetilde{\mu}_1(m_i)>0$
for all $i$ large enough. This implies  $\widetilde{\mu}_1(m_i)
=\frac{1}{\lambda_1(m_i)}\, $ and $\widetilde{u}_{m_i}= u_{m_i}$. Positiveness and pointwise convergence of $u_{m_i}$ to $z$ imply $z\geq 0$ a.e. in $\Omega$.
Moreover, by \eqref{normaliz2} we have
$$ \int_\Omega m_i u^2_{m_i} \, dx =\cfrac{1}{\lambda_1(m_i)}\, $$
and by i), passing to the limit, we find 
$$ \int_\Omega m z^2 \, dx =\cfrac{1}{\lambda_1(m)}\, ,$$
which implies $z\neq 0$. By using \eqref{falena} for $u_{m_i}$  
we have
\begin{equation*}\langle \nabla u_{m_i}, \nabla\varphi\rangle_{L^2(\Omega)} =   
\lambda_1(m_i) \langle m_i u_{m_i}, \varphi\rangle_{L^2(\Omega)} = \lambda_1(m_i)
\int_\Omega  m_i u_{m_i} \varphi \, dx \quad \forall\, \varphi \in H^1(\Omega)
\end{equation*} and, letting $i$ to $\infty$, we deduce $z=u_m$.\\
By i) $\mu_1(m_i) u_{m_i}$     
weakly converges in $H^1(\Omega)$ to $\mu_1(m)u_{m}$ and 
 $\|\mu_1(m_i)u_{m_i}\|_{H^1(\Omega)}
=\mu_1(m_i)$ tends to $\mu_1(m) =\|\mu_1(m) u_{m}\|_{H^1(\Omega)}$.
Hence $\mu_1(m_i) u_{m_i}$
strongly converges to $\mu_1(m)u_{m}$ in $H^1(\Omega)$.\\  
The last claim is immediate provided one observes that $\widetilde{\mu}_1(m)>0$ 
implies $\widetilde{\mu}_1(m_i)>0$ for all $i$ large enough.
\end{proof}

\begin{lemma}\label{teo2} Let $m, q\in L_<^\infty(\Omega)$,     
$\widetilde{\mu}_1(m)$ be defined as in \eqref{muk} for $k=1$. Then\\
i) the map $m\mapsto \widetilde{\mu}_1(m)$      
 is convex on $L_<^\infty(\Omega) $; \\
ii) if $m$ and  $q$ are linearly independent and $ \widetilde{\mu}_1(m),  \widetilde{\mu}_1(q)>0$, then 
\begin{equation*} \widetilde{\mu}_1(tm+(1-t)q)< t \widetilde{\mu}_1(m)+(1-t)  \widetilde{\mu}_1(q)  
\end{equation*}
for all $0<t<1$.\\ 
\end{lemma}

\begin{proof}
i) The Fischer's Principle \eqref{Fischer1} and \eqref{tilde} both for $k=1$ yield
\begin{equation}\label{basta}
\sup_{f\in V_m(\Omega) \atop f\neq 0}
\cfrac{\int_\Omega mf^2\,dx}{\int_\Omega|\nabla f|^2\,dx}\,\leq \widetilde{ \mu}_1(m)  
\end{equation}
for every $m \in L_<^\infty(\Omega)$. Moreover, if $\widetilde{ \mu}_1(m)>0$, then equality sign holds and  
the supremum is attained when $f$ is an eigenfunction of $\widetilde{\mu}_1(m)=\mu_1(m)$.
Let $m, q\in L_<^\infty(\Omega)$, $0\leq t\leq 1$. We show that  
\begin{equation}\label{conv}
\widetilde{\mu}_1(tm + (1-t)q)\leq t \widetilde{\mu}_1 (m) +(1-t) \widetilde{\mu}_1 (q).
\end{equation}
If $\widetilde{\mu}_1(tm + (1-t)q)=0$, \eqref{conv} is obvious. Suppose 
$\widetilde{\mu}_1(tm + (1-t)q)>0$. Then, for all $f\in V_{tm + (1-t)q}(\Omega)$, $f\neq 0$, we 
have  
\begin{equation} \label{A1} 
\begin{split}
\cfrac{\int_\Omega (tm+(1-t)q)f^2\,dx}{\int_\Omega|\nabla f|^2\,dx}   
\,=&\,t\,\cfrac{\int_\Omega mf^2\,dx}{\int_\Omega|\nabla f|^2\,dx}\, 
+(1-t)\, \cfrac{\int_\Omega qf^2\,dx}{\int_\Omega|\nabla f|^2\,dx}\\ 
\leq &\,t\,\cfrac{\int_\Omega mf^2\,dx-\frac{\left(\int_\Omega mf\,dx\right)^2}{\int_\Omega m\,dx}}{\int_\Omega|\nabla f|^2\,dx}\,        
+(1-t)\,\cfrac{\int_\Omega qf^2\,dx-\frac{\left(\int_\Omega qf\,dx\right)^2}{\int_\Omega q\,dx}}{\int_\Omega|\nabla f|^2\,dx}\\          
=&\,t\,\cfrac{\int_\Omega m(P_m(f))^2\,dx}{\int_\Omega|\nabla P_m(f)|^2\,dx}\, 
+(1-t)\, \cfrac{\int_\Omega q(P_q(f))^2\,dx}{\int_\Omega|\nabla P_q(f)|^2\,dx}\\ 
\leq&\, t\widetilde{\mu}_1 (m) +(1-t) \widetilde{\mu}_1 (q),\\   
\end{split}         
\end{equation}
where we used iv) of Proposition \ref{prop1} and \eqref{basta} for $m$ and $q$. Taking the supremum in the left-hand term  of \eqref{A1}   
and using \eqref{basta} again with equality sign,  
we find \eqref{conv}.\\
ii) Arguing by contradiction, we suppose that 
equality holds in \eqref{conv}. We will conclude that $m$ and $q$ are linearly dependent. 
Equality sign in \eqref{conv} implies   
 $\widetilde{\mu}_1(tm + (1-t)q)>0$, then (by \eqref{basta}) the equality also holds in \eqref{A1} with 
$f=u=u_{tm+(1-t)q}$.  We get $\int_\Omega mu\,dx=\int_\Omega qu\,dx=0$, thus $u\in V_m(\Omega)\cap V_q(\Omega)$, and  then  
$$\cfrac{\int_\Omega mu^2\,dx}{\int_\Omega|\nabla u|^2\,dx}\, =\widetilde{\mu}_1(m) \,\quad
 \text{ and } \quad\cfrac{\int_\Omega qu^2\,dx}{\int_\Omega|\nabla u|^2\,dx}\, 
 =\widetilde{\mu}_1(q).$$    
 The simplicity 
of the principal eigenvalue, the positiveness of $u$ and the normalization \eqref{normaliz1} 
imply that    
$u=u_m=u_q$.  
By using  \eqref{falena} with $\lambda=\frac{1}{\tilde{\mu}_1(m)}$ and $\lambda=\frac{1}{\tilde{\mu}_1(q)}$ we have
$$ \langle \nabla u, \nabla\varphi \rangle_{L^2(\Omega)} =\cfrac{1}{ \widetilde{\mu}_1(m)}\, 
\langle m u, \varphi\rangle_{L^2(\Omega)} \quad \forall\, \varphi \in H^1(\Omega) 
$$
and
$$\langle\nabla u, \nabla\varphi\rangle_{L^2(\Omega)} =  \cfrac{1}{ \widetilde{\mu}_1(q)}\, 
\langle q u, \varphi\rangle_{L^2(\Omega)} \quad \forall\, \varphi\in H^1(\Omega) .$$
Taking the difference of these identities we find 
$$\left\langle \left( \cfrac{m}{ \widetilde{\mu}_1(m)}\, -\cfrac{q}{ \widetilde{\mu}_1(q)}\, \right) u, \varphi\right\rangle_{L^2(\Omega)} =0 \quad \forall\, \varphi\in H^1(\Omega), $$
which gives $m\widetilde{\mu}_1(q)-q\widetilde{\mu}_1(m)=0$, i.e.   
 $m$ and $q$ are linearly dependent.\\
\end{proof}
         
\begin{corollary} Let $m_0\in L_<^\infty(\Omega)$,  $\widetilde{\mu}_1(m)$ be defined as in 
\eqref{muk} for $k=1$ and $\overline{\mathcal{G}(m_0)}$ the weak* closure in $L^\infty(\Omega)$ of the class of rearrangements $\mathcal{G}(m_0)$ introduced in 
Definition \ref{class}. Then the map $m\mapsto \widetilde{\mu}_1(m)$ is convex  but not
strictly convex on $\overline{\mathcal{G}(m_0)}$. 
\end{corollary}  
\begin{proof}

By ii) of Proposition \ref{convexity} and Corollary \ref{minore},  we have  that $
\overline{\mathcal{G}(m_0)}$ is convex and $\overline{\mathcal{G}(m_0)}\subset 
L_<^\infty(\Omega)$. Then, by Lemma \ref{teo2}, the map $m\mapsto \widetilde{\mu}_1(m)$ is 
convex on $\overline{\mathcal{G}(m_0)}$.  \\
Applying Proposition \ref{prec2}, we find that the constant function $c=\frac{1}{|\Omega|}
\int_\Omega m_0\,dx$ is in $\overline{\mathcal{G}(m_0)}$. By convexity of $
\overline{\mathcal{G}(m_0)}$, $tm_0+(1-t)c\in \overline{\mathcal{G}(m_0)}$ for every 
$t\in [0,1]$.  From the inequality  
$$tm_0+(1-t)c\leq t\|m_0\|_{L^\infty(\Omega)}+(1-t)c\quad\text{a.e. in }\Omega,$$
we obtain 
$$tm_0+(1-t)c\leq 0\quad\text{ a.e.  in }\Omega \quad \forall\, t\leq \frac{c}{c-\|m_0\|
_{L^\infty(\Omega)}}. $$
 Note that $c/\!\!\left(c-\|m_0\|_{L^\infty(\Omega)}\right)\in (0,1)$. 
Therefore, by \eqref{muk}, we conclude that $\widetilde{\mu}_1(m)=0$ for any 
$m$ in the line segment, contained in $\overline{\mathcal{G}(m_0)}$, that joins $c$ and $$
\frac{c}{c-\|m_0\|_{L^\infty(\Omega)}} m_0+\left(1-\frac{c}{c-\|m_0\|_{L^\infty(\Omega)}}
\right)c=\frac{\|m_0\|_{L^\infty(\Omega)}-m_0}{\|m_0\|_{L^\infty(\Omega)}-c}c.$$
This shows that the map $m\mapsto \widetilde{\mu}_1(m)$ is not
strictly convex.    
\end{proof} 

For the definitions and some basic results on the G\^ateaux differentiability
we refer the reader to \cite{ET}.

\begin{lemma}\label{teo3} 
Let $m\in L_<^\infty(\Omega)$,  $\widetilde{\mu}_1(m)$ be
 defined as in \eqref{muk} for $k=1$ and $u_m$ denote the relative unique positive 
eigenfunction of problem \eqref{p0} normalized as in \eqref{normaliz1}.
Then, the map $m\mapsto  \widetilde{\mu}_1(m)$ 
is G\^ateaux differentiable at any $m$ such that $ \widetilde{\mu}_1(m)>0$, 
with G\^ateaux differential equal to $u_m^2$. In other words,
 for every direction $v\in L^\infty(\Omega)$  
  we have  
\begin{equation}\label{gatto} \widetilde{\mu}_1'(m; v) =\int_\Omega u_m^2 v\, dx. \end{equation}
\end{lemma}
 
\begin{proof}
Let us compute 
$$\lim_{t\to 0} \cfrac{ \widetilde{\mu}_1(m+ t v)- 
\widetilde{\mu}_1(m)}{t}\, .$$
Note that $m+tv\in L_<^\infty(\Omega)$ for $t$ sufficiently small and by i) of Lemma 
\ref{teo1ii}, $\widetilde{\mu}_1(m+ t v)$ converges to
$\widetilde{\mu}_1(m)$ as $t$ goes to zero for any $m\in L_<^\infty(\Omega)$ and $v\in L^\infty(\Omega)$. Therefore,  $\widetilde{\mu}_1(m+ t v)>0$ for $t$ small enough.
The eigenfunctions $u_{m}$ and $u_{m+tv}$ satisfy (see \eqref{falena})  
$$\widetilde{\mu}_1(m)\langle\nabla u_m, \nabla\varphi\rangle_{L^2(\Omega)} =
\langle m u_m, \varphi\rangle_{L^2(\Omega)} \quad \forall\, \varphi\in H^1(\Omega) $$
and   
$$\widetilde{\mu}_1(m+ t v)\langle\nabla u_{m+tv},\nabla \varphi\rangle_{L^2(\Omega)} =
\langle (m+tv) u_{m+tv}, \varphi\rangle_{L^2(\Omega)} \quad \forall\, \varphi\in 
H^1(\Omega). $$ 
By choosing $\varphi=u_{m+tv}$ in the former equation, $\varphi=u_m$ in the latter
and comparing we get     
\begin{equation*}
\widetilde{\mu}_1(m+ t v)
\langle m u_m, u_{m+tv}\rangle_{L^2(\Omega)}=
\widetilde{\mu}_1(m) 
\langle (m+tv) u_{m+tv}, u_m\rangle_{L^2(\Omega)}.
\end{equation*}
Rearranging we find
\begin{equation}\label{rap} \cfrac{\widetilde{\mu}_1(m+ t v)- 
\widetilde{\mu}_1(m)}{t}\, \int_\Omega m \, u_m
u_{m +tv}\, dx =\widetilde{\mu}_1(m) \int_\Omega  u_m
u_{m +tv} v\, dx .\end{equation}
If $t$ goes to zero,  
then by ii) of Lemma \ref{teo1ii} it follows that  
$u_{m+ t v}$ converges to $u_m$ in $H^1(\Omega)$ and therefore
in $L^2(\Omega)$. Passing to the limit      
in \eqref{rap} and using \eqref{normaliz2} we conclude
\begin{equation*}
\lim_{t\to 0} \cfrac{\widetilde{\mu}_1(m+ t v)- \widetilde{\mu}_1(m)}{t}\, =
\int_\Omega u_m^2 v\, dx,
\end{equation*}
i.e. \eqref{gatto} holds.\end{proof}

\begin{theorem}\label{lemmafond}  
Let $m_0\in L_<^\infty(\Omega)$, $\overline{\mathcal{G}(m_0)}$ be the weak* closure in $L^\infty(\Omega)$ of the class of rearrangements $\mathcal{G}(m_0)$ introduced in Definition \ref{class}, $\widetilde{\mu}_1(m)$       
 defined as in \eqref{muk} for $k=1$ and $u_m$ denote the relative unique positive 
eigenfunction of problem \eqref{p0} normalized as in \eqref{normaliz1}. Then \\ 
i) there exists a solution of the problem 
\begin{equation}\label{infclos}
\max_{m\in\overline{ \mathcal{G}(m_0)}}\tilde{\mu}_1(m);  
\end{equation} 
ii) if $|\{m_0>0\}|>0$,  any solution $\check m_1$ of \eqref{infclos} belongs to $\mathcal{G}(m_0)$, more explicitly, we have $\tilde\mu_1(m)<\tilde\mu_1(\check m_1)$ for all $m\in\overline{\mathcal{G}(m_0)}\smallsetminus
\mathcal{G}(m_0)$ (note that, in this case, by Proposition \ref{segnorho}
$\tilde\mu_1(\check m_1)=\mu_1(\check m_1)>0$); \\     
iii) if $|\{m_0>0\}|>0$, for every solution $\check{m}_1\in\mathcal{G}(m_0)$ of \eqref{infclos}
 there exists an increasing function $\psi$ such  
that           
\begin{equation*}\check{m}_1= \psi(u_{\check{m}_1}) \quad \text{a.e. in }\Omega, \end{equation*} where     
$u_{\check{m}_1}$ is 
the positive eigenfunction relative to $\mu_1(\check{m}_1)$ normalized 
as in \eqref{normaliz1}.    
\end{theorem}    
\begin{proof}    
i) By Proposition \ref{minore} we have $\overline{\mathcal{G}(m_0)}\subset L^\infty_<(\Omega)
$.
By iii) of Proposition \ref{cane} and i) of Lemma \ref{teo1ii} respectively, $\overline{\mathcal{G}(m_0)}$ 
is sequentially weakly* compact and the map $m \mapsto \widetilde{\mu}_1(m)$
is sequentially weakly* continuous. Therefore, there exists $\check{m}_1\in 
\overline{\mathcal{G}(m_0)}$ such that 
\begin{equation*}  
\widetilde{\mu}_1(\check{m}_1)=\max_{m\in \overline{\mathcal{G}(m_0)}}\widetilde{ \mu}_1(m) .
\end{equation*} 
ii) Note that, by Proposition \ref{segnorho}, 
the condition $|\{m_0>0\}|>0$ guarantees $\widetilde{\mu}_1(m)>0$ on 
$\mathcal{G}(m_0)$ and then $\widetilde{\mu}_1(\check{m}_1)>0$.   
Let $\check{m}_1$ an arbitrary solution of \eqref{infclos}, let us show
 that $\check{m}_1$ actually belongs to $\mathcal{G}(m_0)$. 
 Proceeding by contradiction,  
 suppose that $\check{m}_1\not \in\mathcal{G}(m_0) $. Then, by iii) of Proposition
 \ref{convexity} and by Definition \ref{libellula}, $\check{m}_1$ is not an extreme point of $\overline{\mathcal{G}(m_0)}$
 and thus there exist $m, q \in \overline{\mathcal{G}(m_0)}$ such that $m\neq q$
 and $\check{m}_1= \frac{m + q}{2}\, $.
 By i) of Lemma \ref{teo2} and, being $\check{m}_1$ a maximizer, we have 
$$\widetilde{\mu}_1(\check{m}_1)\leq   
 \cfrac{\widetilde{\mu}_1(m) +\widetilde{\mu}_1(q)}{2}\,
\leq \widetilde{\mu}_1(\check{m}_1)$$ and
then, equality sign holds. This implies 
$\widetilde{\mu}_1(m) =\widetilde{\mu}_1(q)=\widetilde{\mu}_1(\check{m}_1)>0$,   
that is $m$ and $q$ are maximizers as well. Now, applying ii) of Lemma \ref{teo2}
to $m$ and $q$ with $t=\frac{1}{2}\, $, we conclude that $m$ and $q$ are linearly   
dependent and then, by ii) of Corollary \ref{minore}, $m=q$.             
Thus, we conclude that $\check{m}_1\in \mathcal{G}(m_0)$ and ii) is proved.\\      
iii) Let $\check{m}_1\in \mathcal{G}(m_0)$ a solution of \eqref{infclos}. We prove the claim by using Proposition \ref{Teobart}; more precisely, we show that   
\begin{equation}\label{silvia}        
\int_\Omega 
\check{m}_1 u_{\check{m}_1}^2\, dx> \int_\Omega 
m\, u_{\check{m}_1}^2\, dx\end{equation} 
for every $m\in \overline{\mathcal{G}(m_0)}\smallsetminus \{\check{m}_1\}$.
By exploiting the convexity of $\widetilde{\mu}_1(m)$ (see Lemma \ref{teo2}) and its 
G\^ateaux differentiability in $\check{m}_1$ (see Lemma \ref{teo3})
 we have (for details see \cite{ET})
 \begin{equation} \label{maria}
\widetilde{\mu}_1(m)\geq     
\widetilde{\mu}_1\big(\check{m}_1)
+\int_\Omega (m-
\check{m}_1) u_{\check{m}_1}^2\, dx
\end{equation} 
for all $m\in \overline{\mathcal{G}(m_0)}$.
  
First, let us suppose $\widetilde{\mu}_1(m)< \widetilde{\mu}_1(\check{m}_1)$.
Comparing with \eqref{maria} we find  
\begin{equation*}
\int_\Omega 
(m-\check{m}_1) u_{\check{m}_1}^2\, dx<0  ,\end{equation*}
that is \eqref{silvia}.\\    
Next, let us consider the case $\widetilde{\mu}_1(m)=\widetilde{\mu}_1(\check{m}_1)$,
$m\in \overline{\mathcal{G}(m_0)}\smallsetminus \{\check{m}_1\}$. By i) there are not maximizers of $\widetilde{\mu}_1$ in $\overline{\mathcal{G}(m_0)}\smallsetminus \mathcal{G}(m_0)$, therefore $m \in \mathcal{G}(m_0)$.    
    
Being $\check{m}_1\neq m$, by ii) of Corollary \ref{minore}, they are linearly independent.
Then, ii) of Lemma \ref{teo2} implies 
$$ \widetilde{\mu}_1 \left(\frac{\check{m}_1 + m}{2} \right)<\frac{\widetilde{\mu}_1(
\check{m}_1) +\widetilde{\mu}_1(m)}{2}\, = \widetilde{\mu}_1(
\check{m}_1).$$
Then, as in the previous step, \eqref{maria} with $\frac{\check{m}_1 + m}{2}\,$ in place of 
$m$ yields \eqref{silvia}.\\
This completes the proof.
\end{proof}  

We are now able to prove Theorem \ref{exist}.
\begin{proof}[Proof of Theorem \ref{exist}]         
Being $|\{m_0>0\}|>0$, we have  
$$\lambda_1(m)= \cfrac{1}{\mu_1(m)}\, =\cfrac{1}{\widetilde{\mu}_1(m)}\, $$
for all $m\in\mathcal{G}(m_0)$. Therefore, i) and ii) immediately follow by Theorem 
\ref{lemmafond}.\\
iii) Given      
that $\int_\Omega m_0 \, dx<0$, then, by Proposition \ref{prec2} and Proposition \ref{convexity},
the negative constant function $c=\frac{1}{|\Omega|}\, \int_\Omega m_0 \, dx$ belongs to 
$\overline{\mathcal{G}(m_0)}$. Therefore, by definition of $\widetilde{\mu}_1(m)$, 
$\min_{m \in \overline{\mathcal{G}(m_0)}} \widetilde{\mu}_1(m)=0$  which, in turns, 
being $\mathcal{G}(m_0)$ dense in $\overline{\mathcal{G}(m_0)}$ and 
$\widetilde{\mu}_1(m)$ sequentially weak* continuous, implies   
$\inf_{m \in \mathcal{G}(m_0)}{\mu}_1(m)=0$
and, finally, $\sup_{m \in \mathcal{G}(m_0)} \lambda_1(m)=+\infty$.
\end{proof}

\section{Monotonicity of the minimizers in cylinders}\label{monotonicity}
\noindent   
     
In this section we consider the optimization problem \eqref{infclos0} in cylindrical domains.
Here, by (generalized) cylinder we mean a domain of the type $\Omega=  (0,h) \times \omega\subset  \mathbb{R}^N$, where $h>0$ and $\omega \subset  \mathbb{R}^{N-1}$
is a smooth domain. In the sequel, for $x\in \mathbb{R}^N$ we will write $x= (x_1, x')$, with
$x_1\in\mathbb{R}$ and $x'= (x_2, \ldots, x_N)\in \mathbb{R}^{N-1}$. 
Then, exploiting the notion of monotone decreasing rearrangement, we will able to prove that any 
minimizer of  problem \eqref{infclos0} is monotone with respect to $x_1$. 
For a comprehensive 
survey of the monotone rearrangement we use here, we refer the reader to the work of 
Kawhol  (see \cite{Kaw}) and  Berestycki and Lachand-Robert (see \cite{BeL}).         
In this paper, for sake of simplicity, we choose to define this rearrangement only when the domain
is a cylinder and as particular case of the Steiner symmetrization, in order to deduce easily from it  
some of the properties we need.             
For a brief summary of the Steiner symmetrization see \cite{AC}.    
\begin{definition}\label{pim}
Let $\Omega=  (0,h) \times \omega$ where $h>0$ and $\omega \subset  \mathbb{R}^{N-1}$
is a smooth domain, and $u: \Omega\to \mathbb{R}$ a measurable function bounded from below.
Let $U$ be the extension of $u$ onto $(-h,h) \times \omega$ obtained by reflection with respect
to the hyperplane $\{x\in \mathbb{R}^N: x_1=0\}$  and $U^\sharp$ its Steiner symmetrization.
We define the {\it monotone decreasing rearrangement} $u^\star:\Omega\to\mathbb{R}$ of $u$
to be the restriction of $U^\sharp$ on $\Omega$.  
\end{definition}     
In a similar way it can be defined the {\it monotone increasing rearrangement} $u_\star$. 
Note that if $m\in\mathcal{G}(m_0)$, then $m^\star,m_\star\in\mathcal{G}(m_0)$.   
 \noindent  
From the properties on the Steiner symmetrization (see, for example \cite{AC}) and by
Definition \ref{pim},  we deduce the following first two properties of the monotone decreasing 
rearrangement.\\        
\emph{a}) Let $\Omega=  (0,h) \times \omega$,  $u: \Omega\to \mathbb{R}$ be a measurable 
function bounded from below and $\psi: \mathbb{R}\to\mathbb{R}$ an increasing function.
Then
\begin{equation}\label{crescente}     
(\psi(u))^\star=\psi(u^\star)\quad\text{a.e. in }\Omega. 
\end{equation}
\emph{b}) Let $\Omega=  (0,h) \times \omega$, $u, v:\Omega \to 
\mathbb{R}$ two  bounded measurable functions such that $u,v\in L^2(\Omega)$,
then the \emph{Hardy-Littlewood inequality} holds     
\begin{equation}\label{HL}
\int_\Omega uv\,dx\leq \int_{\Omega}u^\star v^\star\,dx.
\end{equation}
Moreover, from \cite[Theorem 2.8 and Lemma 2.10]{BeL} we have

\emph{c}) Let $\Omega=  (0,h) \times \omega$ and $u\in H^1(\Omega)$ a nonnegative function.
Then $u^\star\in H^1(\Omega^\star)$ and the \emph{P\`olya-Szeg\"o inequality} holds
\begin{equation}\label{PS}
\int_{\Omega}|\nabla u^\star|^2\,dx\leq\int_\Omega|\nabla u|^2\,dx. 
\end{equation} 
More generally, it holds the following
\begin{equation}\label{PSB}  
\int_{\Omega}|\nabla_{x'} u^\star|^2\,dx\leq\int_\Omega|\nabla_{x'} u|^2\,dx,\quad\int_{\Omega}|(u^\star)_{x_1}|^2\,dx\leq\int_\Omega|u_{x_1}|^2\,dx.   
\end{equation}        

The equality case of \eqref{PS} is addressed in Theorem 3.1 of  \cite{BeL}. 

We prove Theorem   \ref{steiner}.  
\begin{proof}[Proof of Theorem \ref{steiner}]  
 Let $\check  m$ be a minimizer of problem \eqref{infclos0}. In what follows  we use the ideas of
 the proof of Theorem 2 in \cite{AC}. By ii) of Theorem \ref{exist}, 
there exists an increasing function $\psi$  such that $\check{m}=\psi(u_{\check m})$ a.e. in $\Omega$, where  $u_{\check{m}}\in V_{\check m}(\Omega)$ denotes the unique positive eigenfunction normalized 
by $\|u_{\check m}\|_{V_{\check m}(\Omega)}= 1$.  Therefore, the monotonicity   
of $\check{m}$ is an immediate consequence of the monotonicity of $u_{\check m}$.  
Hence, it suffices to show that either $u_{\check m}=u_{\check m}^\star$ or 
$u_{\check m}={(u_{\check m})}_\star$.       
By using \eqref{mu1} we find  
\begin{linenomath} 
\begin{equation*} 
\check\mu_1=\mu_1(\check m)=\cfrac{\int_\Omega \check m u_{\check m}^2
\,dx}{\int_\Omega |\nabla u_{\check m}|^2\,dx}\,.    
\end{equation*}
\end{linenomath}
The inequalities \eqref{HL}, property \eqref{crescente} and Definition \ref{def1} yield   
\begin{linenomath}  
\begin{equation*}
 \int_\Omega \check m u_{\check m}^2\,dx\leq\int_\Omega \check m^\star (u_{\check m}^\star)^2\,dx=\int_\Omega \check m^\star\big(P_{\check m^\star}(u_{\check m}^\star)
 \big)^2\,dx+\frac{1}{\int_\Omega \check m^\star\,dx}\left(\int_\Omega \check m^\star u_{\check m}^\star\,dx\right)^2         
\end{equation*}  
and \eqref{PS} and iv) of Proposition \ref{prop1} give 
\begin{equation*}  
\int_\Omega |\nabla u_{\check m}|^2\,dx\geq \int_\Omega |\nabla u_{\check m}^\star|
^2\,dx=\int_\Omega |\nabla P_{\check m^\star}(u_{\check m}^\star)|^2\,dx.    
\end{equation*}   
\end{linenomath}
Note that, $\check m^\star\in\mathcal{G}(m_0)$ (in particular $\int_\Omega \check m^\star\,dx
<0$) and  $P_{\check m^\star}(u_{\check m}^\star)\in V_{{\check m}^\star}$.         
Exploiting \eqref{mu1} and the maximality of $\check\mu_1$ we can write
\begin{equation}\label{rocek}
\begin{split}   
\check\mu_1
&=\cfrac{\int_\Omega \check m u_{\check{m}}^2
\,dx}{\int_\Omega |\nabla u_{\check m}|^2\,dx}
\leq\cfrac{\int_\Omega \check m^\star 
(u_{\check m}^\star)^2\,dx}{\int_\Omega |\nabla u_{\check m}^\star|^2\,dx}
\leq\cfrac{\int_\Omega \check m^\star\big(P_{\check m^\star}(u_{\check m}^\star)
 \big)^2\,dx+\frac{1}{\int_\Omega \check m^\star\,dx}\left(\int_\Omega \check m^\star u_{\check m}^\star\,dx\right)^2 }{\int_\Omega |\nabla P_{\check m^\star}(u_{\check m}^\star)|^2\,dx}\\
&\leq\cfrac{\int_\Omega \check m^\star\big(P_{\check m^\star}(u_{\check m}^\star)
 \big)^2\,dx}{\int_\Omega |\nabla P_{\check m^\star}(u_{\check m}^\star)|^2\,dx}
\leq \cfrac{\int_\Omega \check m^\star (u_{\check m^\star})^2
\,dx}{\int_\Omega |\nabla u_{\check m^\star}|^2\,dx}\,= 
\mu_1(\check m^\star)\leq\check\mu_1.     
\end{split}             
\end{equation}  
Therefore, all the previous inequalities become equalities and yield $\int_\Omega\check m^\star  
u_{\check m}^\star\,dx=0$, which implies  $u_{\check m}^\star\in V_{{\check m}^\star}$, and
\begin{linenomath}   
\begin{equation}\label{K}
\int_\Omega \check m u_{\check m}^2\,dx=\int_\Omega \check m^\star (u_{\check m}^\star)^2
\,dx,\quad\quad\int_\Omega |\nabla u_{\check m}|^2\,dx= \int_\Omega |\nabla u_{\check m}
^\star|^2\,dx.  
\end{equation}   
\end{linenomath}
         
Furthermore, by \eqref{2a}, $u_{\check{m}}^\star$ is an eigenfunction associated to $\lambda_1(\check m^\star)$. By the simplicity of $\lambda_1(\check m^\star)$,
$u_{\check m}^\star$ being positive in $\Omega$ and, by \eqref{K},       
$\|u_{\check m}^\star\|_{V_{\check m^\star}(\Omega)}=\|u_{\check m}\|_{V_{\check m}
(\Omega)}=1$, we conclude that $u_{\check m}^\star=u_{\check m^\star}$.
For simplicity of notation, we put $v=u_{\check m}^\star=u_{\check m^\star}$.
By \eqref{rocek}, $\check m^\star$ is a minimizer of \eqref{infclos0}
and $v$ is the normalized positive eigenfunction associated to
$\lambda_1(\check{m}^\star)=\check\lambda_1$.
Moreover, by ii) of Theorem \ref{exist}, there exists an 
increasing function $\Psi$ such that $\check m^\star=\Psi(v)$ a.e in $\Omega$.      
Thus $v$ satisfies the problem        
\begin{equation}\label{10}
\begin{cases}-\Delta v =\check\lambda_1 \Psi(v) v\quad &\text{in } \Omega,\\
\frac{\partial v}{\partial \nu}=0 &\text{on } \partial\Omega. \end{cases}  
\end{equation}
Let $C_{0,+}^\infty({\Omega})=\{\varphi\in C_0^\infty({\Omega}): \varphi \text{ is 
nonnegative}\}$.\\
 From \eqref{10} in
weak form we have
\begin{linenomath}
$$\int_{\Omega} \nabla v\cdot \nabla \varphi_{x_1}\,dx=\check\lambda_1\int_{\Omega}
\Psi(v)v\,\varphi_{x_1}\,dx\quad\forall \varphi\in C^\infty_{0,+}({\Omega}).$$
\end{linenomath}
Being $v\in W^{2,2}_{\rm loc}(\Omega)$ (see \cite{GT}),   we can rewrite the previous equation as  
\begin{linenomath}
$$-\int_{\Omega} \nabla v_{x_1}\cdot \nabla \varphi\,dx=\check\lambda_1\int_{\Omega}
\Psi(v)v\,\varphi_{x_1}\,dx.$$
\end{linenomath}   
Adding  $\check\lambda_1\int_{\Omega}
\Psi(v)v_{x_1}\,\varphi\,dx$ to both sides and since $v\in C^{1,\beta}(\Omega)$ for all 
$\beta\in(0,1)???$ (see \cite{GT}), 
it becomes   
\begin{equation}\label{E-1}
-\int_{\Omega} \nabla v_{x_1}\cdot \nabla \varphi\,dx+\check\lambda_1\int_{\Omega}
\Psi(v)v_{x_1}\,\varphi\,dx=\check\lambda_1\int_{\Omega} \Psi(v)(v\,\varphi)_{x_1}\,dx.
\end{equation}
Let us show that $\int_{\Omega} \Psi(v)(v\,\varphi)_{x_1}\,dx\geq 0.$
By Fubini's Theorem we get
\begin{equation}\label{E0}
\int_{\Omega} \Psi(v)(v\,\varphi)_{x_1}\,dx=
\int_{\omega}dx'\int_0^h \Psi(v)(v\,\varphi)_{x_1}\,dx_1.
\end{equation}
For any fixed $x'\in\omega$, let $\alpha=\alpha(x_1)=v(x_1,x')\,\varphi(x_1,x')$. 
Since $\varphi$ has compact support, we can consider $\alpha$ trivially defined on 
the whole $[0, h]$.   
Since $\alpha(x_1)$ is continuous and $\Psi(v)$ is decreasing
with respect to $x_1$, the Riemann-Stieltjes integral $\int_0^{h} \Psi(v)\,d\alpha(x_1)$ is well defined (see Theorem 7.27 and the subsequent note in  \cite{A}).
Moreover, by using \cite[Theorem 7.8]{A} we have
\begin{equation}\label{E1}\quad
\int_0^{h} \Psi(v)(v\,\varphi)_{x_1}\,dx_1=\int_0^{h} \Psi(v)\,d\alpha(x_1).
\end{equation}
By \cite[Theorems 7.31 and 7.8]{A} there exists a point $x_0$ in $[0, h]$ such that
\begin{equation*}
\begin{split}
-\int_0^{h}\Psi(v)\,d\alpha(x_1)&=-\Psi(v(0,x'))\int_0^{x_0} \,d\alpha(x_1)-\Psi(v(h, x'))\int_{x_0}^{h} \,d\alpha(x_1)\\
&=-\Psi(v(0,x'))\int_0^{x_0}(v\,\varphi)_{x_1}\,dx_1-\Psi(v(h,x'))\int_{x_0}^{h}(v\,\varphi)_{x_1}\, dx_1.\\
\end{split}
\end{equation*}   
Computing the integrals and recalling that $\varphi\in C^\infty_{0,+}({\Omega})$, $v$ is positive and $\Psi(v)$ is decreasing, we conclude that
$$-\int_0^{h}\Psi(v)\,d\alpha(x_1)=v(x_0,x')\varphi(x_0,x')\left[\Psi(v(h,x'))-\Psi(v(0,x'))\right]\leq 0.$$
Therefore, by the previous inequality and \eqref{E1} it follows
$\int_0^{h} \Psi(v)(v\,\varphi)_{x_1}\,dx_1\geq 0$
for any $x'\in\omega$ and, in turn, from \eqref{E0} we obtain $\int_{\Omega} \Psi(v)(v\,\varphi)_{x_1}\,dx\geq 0$.
Hence, by \eqref{E-1}, $v_{x_1}$ satisfies the differential inequality
\begin{equation*}
\Delta v_{x_1}+\check\lambda_1\Psi(v)v_{x_1}\geq 0\quad \text{in }\Omega
\end{equation*}
in weak form.
Then, applying \cite[Theorem 2.5.3]{PS} and since $v_{x_1}\leq 0$ in $\Omega$, we conclude that either $v_{x_1}\equiv 0$ or
$v_{x_1}<0$.\\ 
In the first case, $v$, and then $u_{\check m}$, is constant with respect to $x_1$.\\
Let $v_{x_1}<0$. By the second equality of \eqref{K} and \eqref{PSB} we obtain
\begin{equation*}   
\int_{\Omega}|(u_{\check m})_{x_1}|^2\,dx=\int_\Omega|{(u_{\check m}^\star)}_{x_1}|^2\,dx.   
\end{equation*}     
 By Fubini's Theorem it becomes
\begin{equation}\label{ultima}
\int_{\omega}dx'\int_0^h |(u_{\check m})_{x_1}|^2\,dx_1=
\int_{\omega}dx'\int_0^h |{(u_{\check m}^\star)}_{x_1}|^2\,dx_1.
\end{equation} 

Being $u_{\check m}$ and $u_{\check m}^\star=u_{\check m\star}$   of class $C^{1,\beta}(\Omega)$,   the functions $x'\mapsto\int_0^h |(u_{\check m})_{x_1}|^2\,dx_1$ and
$x'\mapsto\int_0^h |{(u_{\check m}^\star)}_{x_1}|^2\,dx_1$  are continuous on $\omega$. 
Therefore,  using identity \eqref{ultima} and \eqref{PS} in the one dimensional case we find
\begin{equation*}  
\int_0^h |(u_{\check m})_{x_1}|^2\,dx_1=
\int_0^h |{(u_{\check m}^\star)}_{x_1}|^2\,dx_1\quad\forall x'\in\omega. 
\end{equation*}  
From Theorem 3.1 in \cite{BeL} again in the one dimensional case, we conclude that 
for all $x'\in \omega$ either  $u_{\check m}=u_{\check m}^\star $  or $u_{\check m}
=(u_{\check m})_\star$.  This implies that   for all  $x'\in \omega$ either  $(u_{\check m})_{x_1}
=(u_{\check m}^\star)_{x_1}=v_{x_1}<0$ or $(u_{\check m})_{x_1}=((u_{\check 
m})_\star)_{x_1}=-(u_{\check m}^\star)_{x_1}=-v_{x_1}>0$.  Being $(u_{\check m})_{x_1}$
continuous in $\Omega$ and $\Omega$ an open connected set, it follows that $(u_{\check 
m})_{x_1}$   does not change sign in $\Omega$, equivalently either  $u_{\check m}=u_{\check 
m}^\star $  or $u_{\check m}=(u_{\check m})_\star$ in the whole $\Omega$.        
Finally, by $\check{m}=\psi(u_{\check m})$,  we conclude that $\check m=\check m^\star$ or $\check m=\check m_\star$. This proves the theorem.
\end{proof}

\noindent \textbf{Acknowledgments}.
The authors are partially supported by the research project {\em Analysis of PDEs in connection 
with real phenomena}, CUP F73C22001130007, funded by \href{https://
www.fondazionedisardegna.it/}{Fondazione di Sardegna}, annuity 2021.   
The authors are members of GNAMPA (Gruppo Nazionale per l'Analisi Matematica, la 
Probabilit\`a e le loro Applicazioni) of INdAM (Istituto Nazionale di Alta Matematica
``Francesco Severi").\\
The authors also acknowledge the financial support under the National Recovery and Resilience Plan (NRRP), Mission 4 Component 2 Investment 1.5 - Call for tender No.3277 published on December 30, 2021 by the Italian Ministry of University and Research (MUR) funded by the European Union – NextGenerationEU. Project Code ECS0000038 – Project Title eINS Ecosystem of Innovation for Next Generation Sardinia – CUP F53C22000430001- Grant Assignment Decree No. 1056 adopted on June 23, 2022 by the Italian Ministry of University and Research (MUR).


\begin{thebibliography}{99}

\bibitem{Alvino}
A. Alvino, G. Trombetti and P.L. Lions, \emph{On optimization 
problems with prescribed rearrangements}, Nonlinear Anal. \textbf{13}, no. 2 (1989), 185-220.\\  https://doi.org/10.1016/0362-546x(89)90043-6 

\bibitem{AC}
C. Anedda and F. Cuccu, \emph{Optimal location of resources and Steiner symmetry in a population dynamics model in heterogeneous environments}, Annales Fennici Mathematici, \textbf{47} (1),
 (2022) 305-324. 

\bibitem{ACF}
C. Anedda, F. Cuccu and S. Frassu, \emph{Steiner symmetry in the minimization of the first eigenvalue of a fractional eigenvalue problem with indefinite weight}, Can. J. Math. \textbf{73},
no. 4 (2021),  970-992. 
https://doi.org/10.4153/S0008414X20000267    

\bibitem{A} 
T.M. Apostol, \emph{Mathematical Analysis}, Addison-Wesley (1974). 

 \bibitem{BHR}
 H. Berestycki, F. Hamel and L. Roques, \emph{Analysis of the periodically fragmented environment model: I – Species persistence}, J. Math. Biol. \textbf{51} (2005),  75-113. 
 https://doi.org/10.1007/s00285-004-0313-3

 \bibitem{BeL} H. Berestycki and T. Lachand-Robert \emph{Some properties of monotone rearrangement with applications to elliptic equations in cylinders}, Math. Nachr. {\bf 266}, no. 1 (2004), 3-19. https://doi.org/10.1002/mana.200310139
 
\bibitem{Bo}
M. B\^{o}cher, \emph{The smallest characteristic numbers in a certain exceptional case}, Bull. Am. Math. Soc. \textbf{21} (1), (1914) 6-9.  

\bibitem{BL}
K.J. Brown and C.C. Lin, \emph{On the existence of positive eigenfunctions for an eigenvalue problem with indefinite weight-function}, Math. Anal. Appl. \textbf{75} (1980), 112-120.

\bibitem{B}
G.R. Burton, \emph{Rearrangements of functions, maximization of convex functionals and vortex rings}, Math. Ann. \textbf{276} (1987), 225-253. 
https://doi.org/10.1007/bf01450739  

\bibitem{B89}
G.R. Burton, \emph{Variational problems on classes of rearrangements and multiple configurations for steady vortices}, 
Ann. Inst. H. Poincar\'e Anal. Non Lin\'eaire \textbf{6}, no. 4, (1989), 295-319. 
https://doi.org/10.1016/S0294-1449(16)30320-1

\bibitem{CFP}
L. Cadeddu, M.A. Farina and G. Porru, \emph{Optimization of the principal eigenvalue under mixed boundary conditions}, 
Electron. J. Diff. Equ. \textbf{154}, (2014), 1-17. 

\bibitem{CC}
R.S. Cantrell and C. Cosner, \emph{Diffusive logistic equations with indefinite weights: population models in disrupted environments}, Proc. R. Soc. Edinb. Sect A \textbf{112}, 3-4 (1989),  293-318. 
https://doi.org/10.1017/s030821050001876x

\bibitem{CC91}
R.S. Cantrell and C. Cosner, \emph{The effects of spatial heterogeneity in population dynamics}, J. Math. Biol. \textbf{29} (1991),  315-338. 
https://doi.org/10.1007/BF00167155

\bibitem{CC03}
R.S. Cantrell and C. Cosner, \emph{Spatial ecology via reactio-diffusion equations}, Wiley Series
in Mathematical and Computational Biology, Wiley, Chichester  (2003).

\bibitem{CGIKO}
S. Chanillo, D. Grieser, M. Imai, K. Kurata and I. Ohnishi, \emph{Symmetry breaking and other phenomena in the
optimization of eigenvalues for composite membranes}, Commun. Math. Phys. \textbf{214}, 2 (2000),  315-337. 
https://doi.org/10.1007/PL00005534

\bibitem{CML1}
S.J. Cox and J.R. McLaughlin, \emph{Extremal eigenvalue problems for composite membranes, I}, Appl. Math. Optim. \textbf{22} (1990), 153-167. 
https://doi.org/10.1007/bf01447325

\bibitem{CML2}
S.J. Cox and J.R. McLaughlin, \emph{Extremal eigenvalue problems for composite membranes, II}, Appl. Math. Optim. \textbf{22} (1990), 169-187. 
https://doi.org/10.1007/bf01447326

\bibitem{day70} 
P. W. Day,  \emph{Rearrangements of measurable functions}, Dissertation (Ph.D.), California Institute of Technology (1970).
 
\bibitem{DF}  
D.G. de Figueiredo, 
\emph{Positive solutions of semilinear elliptic problems}, in A. Dold, B. Eckmann (eds), Differential equations, Lecture Notes in Mathematics \textbf{957}, Springer 
(1982), 34-87. 
https://doi.org/10.1007/BFb0066233

\bibitem{DGT} 
A. Derlet, J.-P. Gossez and P. Tak\`a\v{c}, \emph{Minimization of eigenvalues for a quasilinear elliptic Neumann problem with indefinite weight}, J. Math. Anal. Appl. \textbf{371} (1) (2010), 69-79. 

\bibitem{DPLV} 
S. Dipierro, E. Proietti Lippi and E. Valdinoci, \emph{(Non)local logistic equations with Neumann 
conditions}, Ann. Inst. H. Poincaré Anal. Non Linéaire  (2022), DOI 10.4171/AIHPC/57.

\bibitem{ET} 
I. Ekeland and R. T\'emam, \emph{Convex analysis and variational problems}, Classics in Applied Mathematics  \textbf{28}, SIAM (1999). 

\bibitem{F} 
W.H. Fleming, \emph{A selection-migration modelin population genetics}, J. Math. Biol. \textbf{2} (3) (1975), 219-233.

\bibitem{GT} 
D. Gilbarg and N.S. Trudinger, \emph{Elliptic Partial Differential Equations of Second Order}, Springer, Berlin, Heidelberg (1977).
https://doi.org/10.1007/978-3-642-96379-7

\bibitem{hardy52}
G. H. Hardy, J. E. Littlewood and G. P\'olya, \emph{Inequalities}, Cambridge University Press, second edition (1952, first published in 1934).

\bibitem{He} 
A. Henrot, \emph{Extremum problems for eigenvalues of elliptic operators}, Frontiers in Mathematics, Birkhuser Verlag, Basel (2006).

\bibitem{HKL}
M. Hinterm\"uller, C.-Y. Kao and A. Laurain, \emph{Principal eigenvalue minimization for
an elliptic problem with indefinite weight and Robin boundary conditions},  Appl. Math. Optim. 
\textbf{65} (1), 111–146 (2012).

\bibitem{JP}
 K. Jha and G. Porru, \emph{Minimization of the principal eigenvalue under
   Neumann boundary conditions}, Numer. Funct. Anal. Optim. \textbf{32}, 11 (2011), 1146-1165. 
 https://doi.org/10.1080/01630563.2011.592244

\bibitem{KLY}
C.-Y Kao, Y. Lou and E. Yanagida, \emph{Principal eigenvalue for an elliptic problem with indefinite weight on  
cylindrical domains}, Math. Biosci. Eng. \textbf{5} (2) (2008), 315-335. 

\bibitem{Kaw}
B. Kawhol, \emph{Rearrangements and convexity of level sets in PDE}, Lecture Notes in Mathematics vol. \textbf{1150}, Springer, Berlin (1985).
 
\bibitem{LLNP}
J. Lamboley, A. Laurain, G. Nadin and Y. Privat, \emph{Properties of optimizers of the principal eigenvalue
with indefinite weight and Robin conditions}, Calc. Var. Partial Differential Equations \textbf{55}, 144 (2016), 1-37. 
https://doi.org/10.1007/s00526-016-1084-6

\bibitem{Lax} P.D. Lax, \emph{Functional analysis}, Wiley (2002).

\bibitem{L}
G. Leoni, \emph{A First Course in Sobolev Spaces}, AMS (2009).

\bibitem{LY}
Y. Lou and E. Yanagida, \emph{Minimization of the Principal Eigenvalue for an Elliptic Boundary Value Problem with Indefinite Weight,
  and Applications to Population Dynamics}, Jpn J. Indust. Appl. Math. \textbf{23}, 275 (2006), 275-292. 
https://doi.org/10.1007/BF03167595

\bibitem{MNP}
I. Mazari, G. Nadin and Y. Privat, \emph{Some challenging optimization problems for logistic diffusive equations and their numerical modeling}, Handbook of Numerical Analysis, \textbf{23}, (2022), 401-426.

\bibitem{Mik}
V.P. Mikhailov, \emph{Partial differential equations}, Mir Publishers, Moscow (1978).
https://doi.org/10.1016/bs.hna.2021.12.012 


\bibitem{PV}
B. Pellacci and G. Verzini, \emph{Best dispersal strategies in spatially heterogeneous environments: optimization of the principal eigenvalue for indefinite fractional Neumann problems}, J. Math. Biol.  \textbf{76} (2018), 1357-1386.
https://doi.org/10.1007/s00285-017-1180-z

\bibitem{PS}  
 P. Pucci and J. Serrin, 
\emph{Maximum Principles for Elliptic Partial Differential Equations} in 
 Handbook of Differential Equations: Stationary Partial Differential Equations, Vol. \textbf{4}, Ch. 6, Edited by M. Chipot, Elsevier BV (2007), 355-483. 
 https://doi.org/10.1016/S1874-5733(07)80009-X
 
 \bibitem{ryff67}
J. V. Ryff, \emph{Extreme points of some convex subsets of $L^1(0,1)$}, Proc. Amer. Math. Soc. \textbf{18} (1967), 1026-1034. 
https://doi.org/10.1090/s0002-9939-1967-0217586-3

 \bibitem{RH}
 L. Roques and F. Hamel, \emph{Mathematical analysis of the optimal habitat
   configurations for species persistence}, Math. Biosci. \textbf{210}, 1 (2007), 34-59. 
 https://doi.org/10.1016/j.mbs.2007.05.007

\bibitem{SH}
S. Senn and P. Hess, \emph{On positive solutions of a linear elliptic boundary value problem with
Neumann boundary conditions}, Math. Ann., \textbf{258} (1982), 459-470.

\bibitem{S}
J.G. Skellam, \emph{Random Dispersal in Theoretical Populations}, Biometrika \textbf{38}, 1/2 (1951), 196-218. 
https://doi.org/10.2307/2332328

\end{thebibliography}
\end{document}